\documentclass[a4paper,11pt]{amsart}
\usepackage{amsmath,amsthm,amssymb,amsfonts,enumerate,color}
\usepackage[pdftex]{graphicx}

\usepackage{verbatim}

\usepackage{color}
\usepackage{tikz}
\usepackage{graphicx}
\newcommand{\abs}[1]{\left\vert#1\right\vert}

\newtheorem{theorem}{Theorem}[section]

\newtheorem{lemma}[theorem]{Lemma}

\newtheorem{proposition}[theorem]{Proposition}
\newtheorem{definition}[theorem]{Definition}

\newtheorem{remark}[theorem]{Remark}

\numberwithin{equation}{section}

\newcommand{\edproof}{ $\hfill {\Box}$}

\allowdisplaybreaks

\newcommand{\R}{\mathbb{{R}}}

\newcommand{\N}{\mathbb{{N}}}

\newcommand{\LL}{\mathcal{{L}}}
\newcommand{\HH}{\mathcal{{H}}}

\author[De Le\'on-Contreras]{Marta De Le\'on-Contreras}
\address{Department of Mathematics and Statistics, University of Reading, RG6 6AX Reading, United Kingdom.}
\email{m.deleoncontreras@reading.ac.uk}

\author[Torrea]{Jos\'e L. Torrea}
\address{Departamento de Matem\'aticas, Facultad de Ciencias, Universidad Aut\'onoma de Madrid, 28049 Madrid, Spain.}
\email{joseluis.torrea@uam.es}

\thanks{First autor was partially supported by grant 
 EPSRC Research Grant EP/S029486/1.  Second autor was  partially supported by grant  PGC2018-099124-B-I00 (MINECO/FEDER)}


\begin{document}
	
	\title{Lipschitz spaces adapted to Schr\"odinger operators and regularity properties.}
	\begin{abstract}

	Consider  the  Schr\"odinger operator $\LL=-\Delta+V$
		in $\R^n, n\ge 3,$  where $V$ is a nonnegative potential satisfying 
		a reverse H\"older condition of the type 
		\begin{equation*}
		\left( \frac{1}{|B|}\int_B V(y)^qdy\right)^{1/q}\le \frac{C}{|B|}\int_B V(y)dy, \,  \text{{ for some }}q>n/2.
		\end{equation*}
			
		 We define 
				 $\Lambda^\alpha_\LL,\,  0<\alpha <2,$  the class of measurable  functions such that 
			$$ \|\rho(\cdot)^{-\alpha}f(\cdot)\|_\infty<\infty \quad  \, \, \text{and}\:\:
			\quad  \sup_{|z|>0}\frac{\|f(\cdot+z)+f(\cdot-z)-2f(\cdot)\|_\infty}{|z|^\alpha}<\infty,
			$$
			where  $\rho$ is the critical radius function associated to $\mathcal{L}$.

					Let $W_y f = e^{-y\LL}f$ be  the heat semigroup of $\LL$. Given  $\alpha >0,$ we denote by $\Lambda_{\alpha/2}^{{W}}$ the set of functions $f$ which satisfy  \begin{equation*} \|\rho(\cdot)^{-\alpha}f(\cdot)\|_\infty<\infty \hbox{ and } \Big\|\partial_y^k{W}_y f \Big\|_{L^\infty(\mathbb{R}^{n})}\leq C_\alpha y^{-k+\alpha/2},\;\: \, {\rm with }\, k=[\alpha/2]+1, y>0.  
				\end{equation*}
					
					We prove that for $0<\alpha \le 2-n/q$,  $\Lambda^\alpha_\LL = \Lambda_{\alpha/2}^{{W}}.$  
					
												As application,	we obtain regularity properties of fractional powers (positive and negative) of the operator $\LL$, Schr\"odinger Riesz  transforms, Bessel potentials and multipliers of Laplace transforms type. The proofs of these results need in an essential way the language of semigroups.

					Parallel results are obtained for the classes defined through
					the Poisson semigroup, $P_yf= e^{-y\sqrt{\LL}}f.$

		\end{abstract}
		
		\maketitle
		
		\section{Introduction} 
		
			Classical Lipschitz spaces on $\R^n$, $\Lambda^\alpha,$ $\alpha>0$, are classes of smooth functions that play an important role in function theory, harmonic analysis and partial differential equations.  For $0< \alpha <1$, they are defined as the set of functions $\varphi$  such that  $ |\varphi(x+z)-\varphi(x)| \le C |z|^\alpha$ $x,z\in\R^n$. If the functions are supposed to be bounded, it is also usual to call them  H\"older functions and denote the class by $C^{0,\alpha}$.  For $k \in \mathbb{N}$ {and} $0 < {\beta} <1,$    $C^{k,{\beta} }$, see \cite{Silvestre},  are defined as the classes  of functions such that the derivatives of order less or equal to $k$ are continuous and bounded, while the derivatives of order $k$ belong to $C^{0,{\beta} }$. Sometimes these classes are denoted by  $C^{k+{\beta} }$, see \cite{Krylov book}. When $\alpha \ge 1$, some definitions of the classes  $\Lambda^\alpha$ can be found in the literature, through finite differences, see \cite{Krantz}, and also through some integral estimates, see \cite{Stein, Taibleson}. Moreover, for  $\alpha \notin  \mathbb{N}$ {such that  $\alpha=k+\beta$}, $k\in\N$, $0<\beta<1$, the classes   $\Lambda^\alpha$ and $C^{{k,\beta} }$ agree.
						
			This paper is doubly motivated by the works of E. Stein and M. Taibleson, \cite{Stein, Taibleson}, Krylov, \cite{Krylov book}, and  Silvestre, \cite{Silvestre}. {On the} one hand, 
			in \cite{Stein, Taibleson} the authors characterized  the classes of  bounded Lipschitz functions $\Lambda^\alpha$ by integral estimates of the  Gauss {semigroup}, $e^{y{\Delta}}$, and the Poisson semigroup, $e^{-y\sqrt{-\Delta}}.$ {On the other hand,} in  \cite{Krylov book}
			and \cite{Silvestre} the authors studied  the boundedness of different operators associated to $\Delta$ in the classes $C^{k,{\beta}}$. 
			Our  aim is to {analyze} the above works in the case of    Schr\"odinger operators $\LL=-\Delta+V$,
		in $\R^n$ with $n\ge 3$,  where $V$ is a nonnegative potential satisfying a  reverse H\"older inequality, see (\ref{reverseHolder}).
		
		Namely, our purposes are the following: 
		\begin{itemize}
		\item  To find the appropriated  point-wise  definition of  Lipschitz  classes in the Schr\"odinger setting for $0< \alpha <2$.  We shall denote this space by $\Lambda^\alpha_\LL$. 
		\item To characterize these classes by using either the heat semigroup, $e^{-y\LL}$, or the Poisson semigroup, $e^{-y\sqrt{\LL}}$. 
		\item To use these  characterizations to prove  H\"older estimates of  negative and positive powers of the operator $\LL$. We shall also show the boundedness of Bessel potentials, Riesz  transforms  and some multiplier operators associated to $\LL$.  We remark that we don't need the point-wise expression of the operators.
		\end{itemize}

		Now we present our definitions and  results.
		
					\begin{definition} Let $0 < \alpha < 2$ and $\rho(x)$    the critical radius {function}, see \eqref{critical}.  We shall denote by 
				 $\Lambda^\alpha_\LL$ the class of measurable  functions such that 
			$$M^\LL_\alpha[f] := \|\rho(\cdot)^{-\alpha}f(\cdot)\|_\infty<\infty \quad  \, \, \text{and}\:\:
			\quad N_\alpha[f]:= \sup_{|z|>0}\frac{\|f(\cdot+z)+f(\cdot-z)-2f(\cdot)\|_\infty}{|z|^\alpha}<\infty. 
			$$	
			We endow this space with the norm $$\|f\|_{\Lambda^\alpha_\LL} := M^\LL_\alpha[f]+ N_\alpha[f].$$

\end{definition}
\noindent We shall see that  $\Lambda^\alpha_\LL$ coincides the space defined in \cite{BonHar}  for $0<\alpha <1,$ \eqref{polabruno}, see Remark \ref{espaciopola}.	
	
\begin{remark} The set of continuous functions  $\varphi$ such that $ |\varphi(x+z)+\varphi(x-z)-2\varphi(x)  | \le C |z|$, $x,z\in\R^n$ was introduced by Zygmund, see \cite{Zygmund}, {and it is usually} called Zygmund's class.
\end{remark}

By ${W}_y = e^{-y{\LL}}$  we will denote the heat  semigroup associated to 	 $\LL.$  From the Feynman-Kac formula, it is well known that 
$${W}_y (x,z) \le (4\pi y)^{-n/2} e^{ -\frac{|x-z|^2}{4y}}.$$
Motivated by this estimate, we shall say that a  function $f$ satisfies a {\bf heat size condition for } $\LL$ if 
 $\displaystyle \int_{\mathbb{R}^n}e^{-\frac{|x|^2}{y}} |f(x)| dx < \infty, $ for every $y>0$,  and  for every $\ell\in \N\cup \{0\},$  and every $x\in \mathbb{R}^n,\,\lim_{y\rightarrow \infty}  \partial_y^\ell{W}_yf(x)  =0.$ When some estimates on the derivatives of the heat semigroup are assumed, the following theorem shows that this {\bf heat size condition} is equivalent to a controlled growth of the function.

	\begin{theorem}\label{tam2} 
						Let {$\alpha>0$} and  $f$ be a measurable function. The following are equvivalent:
						\begin{itemize}
							\item   $f$ satisfies a  {\bf heat size condition for $\LL$} and
					
				\begin{equation}\label{derivada} \Big\|\partial_y^k{W}_y f \Big\|_{L^\infty(\mathbb{R}^{n})}\leq C_\alpha y^{-k+\alpha/2},\;\: \, {\rm with }\,\: k=[\alpha/2]+1, y>0.  
				\end{equation}
				\item $f$ satisfies  $M^\LL_\alpha[f]< \infty$ and \eqref{derivada}.	
 
 	\end{itemize}
					\end{theorem}
This theorem leads  us to the next definition. 
\begin{definition} Let $\alpha >0.$ We shall denote by  $\Lambda_{\alpha/2}^{{W}}$ the set of functions $f$ which satisfy  a  {\bf heat size condition for } $\LL$ and 
\eqref{derivada}.
We endow this space with the norm 
$$\|f\|_{\Lambda^W_{\alpha/2}} := S^W_\alpha[f]+   M^\LL_\alpha[f],$$
being
$S^W_\alpha[f]$  the infimum of the constants $C_\alpha$ appearing in 
					\eqref{derivada}.
			\end{definition}  
	Now we state the announced characterization of the Lipschitz classes by using the  derivatives of the heat semigroup.		
					\begin{theorem}\label{identities} Let $0<\alpha\le 2-\frac{n}{q}.$ Then    
						$$\Lambda^\alpha_\LL= \Lambda_{\alpha/2}^{W},$$
					with equivalence of their norms.
							\end{theorem}
	 	  
 Some observations are in order. The restriction in the range $0< \alpha \le 2-\frac{n}{q}$ is due to the reverse H\"older inequality \eqref{reverseHolder} that satisfies the potencial $V$. If the potential $V$ satisfies \eqref{reverseHolder} for every $q>n/2$, then we get the result for every $0<\alpha<2$. This is the case of the Hermite operator, $\HH=-\Delta+|x|^2$. To prove Theorem \ref{identities}, we compare the spaces $\Lambda_{\alpha/2}^{W}$ with some parallel spaces $\Lambda_{\alpha/2}^{\tilde{W}}$ defined for the classical Laplace operator, see Definition \ref{defclasica}. We believe that these spaces, more general  that the classical  Lipschitz spaces, are of independent interest. The functions don't need to be bounded, however a point-wise characterization is also valid as in the classical case, see Theorem \ref{nuevoStein}.    Once we have this characterization, by using the so called ``perturbation formula'' for Schr\"odinger operators,  we get a comparison between the classes $\Lambda_{\alpha/2}^{\tilde{W}}$ and $\Lambda_{\alpha/2}^{{W}}$, see Theorem \ref{comparacion}.

 	 Lipschitz spaces adapted to the operator $\LL$ have been analyzed by different authors for $0<\alpha<1$. In the paper \cite{BonHar} the authors introduced, for $0<\alpha< 1$, the space of functions which satisfy  
		\begin{equation}\label{polabruno} \|\rho(\cdot)^{-\alpha}f(\cdot)\|_\infty<\infty \quad  \, \, \text{and}\:\:
		\quad \sup_{|z|>0}\frac{\|f(\cdot-z)-f(\cdot)\|_\infty}{|z|^\alpha}<\infty, \end{equation}
		where 	$\rho(x)$  is  the critical radius function associated to the potential $V$, see \eqref{critical}. See also \cite{MSTZ}.
	In the case of the Hermite operator, $\mathcal{H}=-\Delta+|x|^2$, adapted H\"older classes were defined point-wise in \cite{ST2}.    By using the Poisson semigroup some parabolic classes were considered   in   \cite{ST3}. {Finally}, for the Ornstein-Ulhenbeck operator, $\mathcal{O}=-\frac{1}{2}\Delta+x\cdot\nabla$, in \cite{Urbina} some Lipschitz classes were defined by means of its Poisson semigroup, $e^{-y\sqrt{\mathcal{O}}}$, and in \cite{Sjogren} a point-wise characterization was obtained for $0<\alpha<1$.  Our Theorem \ref{identities} contains as particular cases the results in  \cite{BonHar} and \cite{MSTZ}, when $0<\alpha <1$.	 In the case of Hermite operator, the fact that 
	$\mathcal{H} = -\Delta+|x|^2= \sum_i A_iA_{-i}+ A_{-i}A_i$
with 
	 $A_{\pm}i = \pm\partial_{x_i} + x_i $ can be used to define spaces $C^{k,{\beta}}_{\mathcal{H}} = \Lambda^{k+{\beta}}_\mathcal{H} , 
	 k\in \mathbb{N}, 0< {\beta} <1$. We shall develop the theory of {those} spaces in a forthcoming paper.
	 
As we said, our third purpose  is to study the regularity of the following operators in the Lipschitz spaces  {previously defined}. 
	\begin{itemize}
	\item  The {\it Bessel potential of order $\beta >0$}, 
	  $$
	 (Id+\LL)^{-\beta/2}f(x)=\frac{1}{\Gamma(\beta/2)}\int_0^\infty e^{-t}e^{-t\LL}f(x) t^{\beta/2}\frac{dt}{t}.
	  $$
	\item The {\it fractional integral} of order $\beta>0$.
	  $$
	  \LL^{-\beta/2}f(x)=\frac{1}{\Gamma(\beta/2)}\int_0^\infty e^{-t\LL}f(x) t^{\beta/2}\frac{dt}{t}.
	  $$
\item The {\it fractional ``Laplacian''} of order $\beta/2>0$ 
	   $$
	   \LL^{\beta/2}f(x)= \frac1{c_\beta} \int_0^\infty
	   (Id-e^{- t\LL})^{[\beta/2]+1}f(x)\,\frac{dt}{t^{1+\beta/2}}.
	   $$
	   
	   \item The first order Riesz transforms  defined by
	  $$\mathcal{R}_i=\partial_{x_i}(\LL^{-1/2}), \hbox{  and } \,  R_i = \LL^{-1/2}(\partial_{x_i}),  \: i=1,\dots,n. $$   
	  \end{itemize}
	  The following theorems  will be proved in Section \ref{pruebaaplicaciones}.
	  \begin{theorem}\label{Schau}
	Let $\alpha,\beta>0$ and $\mathcal{T}_\beta$ denote the Bessel potential or the fractional integral of order $\beta$. Then, $\mathcal{T}_\beta$  satisfies 
\begin{itemize}
\item[(i)]
		$\|\mathcal{T}_\beta f\|_{\Lambda_{\frac{\alpha+\beta}{2}}^{W}}\le C \|f\|_{	\Lambda_{\alpha/2}^{W}}.$
\item[(ii)] $\|\mathcal{T}_\beta f\|_{\Lambda_{{\beta/2}}^{W}}\le C \|f\|_{\infty}.$		
		\end{itemize}
	  \end{theorem}
	
	  	  \begin{theorem}[H\"older estimates]\label{Holderestimates}
	  	Let $0< \beta < \alpha $ and  $f\in	\Lambda_{\alpha/2}^{W}.$  Then, 
	  $$	\|\mathcal{L}^{\beta/2} f\|_{\Lambda_{{\frac{\alpha-\beta}{2}}}^{W}}\le C \|f\|_{\Lambda_{\alpha/2}^{W}}.   
	  	$$
	  \end{theorem}
	  
	  	  \begin{theorem} \textcolor{white}{} \label{TRiesz}
	  	\begin{itemize}
		\item For  $1<\alpha\le 2-n/q$,  then $\|{R}_if\|_{\Lambda_{\alpha/2}^W} \le C \|f\|_{\Lambda_{\alpha/2}^W} $,  $i=1,\dots,n$.
	  		\item For $0<\alpha\le 1-n/q$,  then $\|\mathcal{R}_if\|_{\Lambda_{\alpha/2}^W} \le C \|f\|_{\Lambda_{\alpha/2}^W} $,  $i=1,\dots,n$.
	  		\end{itemize}	
	  	
	  \end{theorem}

	  \begin{theorem}\label{multiplicador}
	  	Let $a$ be a measurable bounded function on $[0,\infty)$ and consider
	  	$$
	  	m(\lambda)=\lambda\int_0^\infty e^{-s\lambda}a(s)ds, \, \, \lambda >0.
	  	$$
	  	Then, for every $\alpha>0$, the multiplier operator of the Laplace transform type $m(\LL)$ is bounded from $\Lambda^{W}_{\alpha/2}$ into itself.
	  \end{theorem}

		 \

There are some important differences when we want to define Lipschitz spaces through the Poisson semigroup. It  can be defined by the following subordination formula
				 \begin{align}\label{Poissonformula}
					 P_yf(x)&=e^{-y\sqrt{\LL}}f(x)=\frac{y}{2\sqrt{\pi}}\int_0^\infty e^{-\frac{y^2}{4 \tau}}e^{-\tau \LL}f(x)\frac{d\tau}{\tau^{3/2}} .  
					 \end{align}
					 Getting inside the Feynman-Kac estimate of the heat kernel we get that  the kernel of the Poisson semigroup, $P_y(x,y)$ satisfies 
					 $$P_y(x,z) \le C \frac{y}{(|x-z|+ y)^{n+1}}.$$
Hence, parallel to the heat semigroup case,  we shall say that a function $f$ satisfies a {\bf Poisson size condition  for $\LL$} if 
				\begin{equation*} M^P[f]:=  \int_{\mathbb{R}^n} \frac{| f(x)|}{(1+|x|)^{n+1} }\, dx < \infty.\end{equation*}

 \begin{theorem}\label{tam3} 
						Let {$\alpha>0$} and $f$ be a function satisfying a  {\bf Poisson size condition for $\LL$} and
\begin{equation}\label{derivada3} \Big\|\partial_y^k{P}_y f \Big\|_{L^\infty(\mathbb{R}^{n})}\leq C_\alpha y^{-k+\alpha},\;\: \, {\rm with }\, k=[\alpha]+1, y>0. 
							\end{equation} 
Then,
 $M^\LL_\alpha[f]< \infty.$  
					\end{theorem}	
					
					\begin{remark}\label{excep}
					Observe that if $f$ is a function such that  $M^\LL_\alpha[f]< \infty$ with  $0<\alpha<1$ (see Lemma \ref{Shenlemma} with $x=0$) or $\rho\in L^\infty(\R^n)$, then  $f$ satisfies a  {\bf Poisson size condition for $\LL$}. 
					\end{remark}
\noindent The previous theorem drives us to the following definition.
	
					\begin{definition}[]\label{PoissonZygmund} Let $f$ be a function that satisfies $M^P[f] <\infty.$  Given $\alpha >0,$ we shall say that $f$ belongs to the class $\Lambda_{\alpha}^{{P}}$ 
if it satisfies \eqref{derivada3}. The linear space  can be endowed with the norm
					\begin{equation*}
			\|f\|_{\Lambda_{\alpha}^{{P}}}:=   S^P_\alpha[f] +M^\LL_\alpha[f],
					\end{equation*}
					where 
					$S^P_\alpha[f]$ is the infimum of the constants $C_\alpha$ appearing in 
					\eqref{derivada3}.
			\end{definition}	
					In \cite{MSTZ}, the authors proved a characterization of the class $\Lambda_\alpha^P$ in the case  $0<\alpha <1$ for functions satisfying the integrability condition   $\displaystyle \int_{\mathbb{R}^n} \frac{|f(z)|}{(|z|+1)^{n+\alpha+\varepsilon}} dz<\infty$.    We can extend the characterization beyond $1$. Namely, we have the following result.			
		 	 \begin{theorem}\label{identities4} Let f be a function with $M^P[f] < \infty.$
			 For  $0< \alpha \le 2-n/q$,  the following statements are equivalent:
				
				$$ f \in \Lambda^\alpha_\LL, \quad 
 f \in \Lambda_{\alpha}^{P}, \quad 
  f \in \Lambda_{\alpha/2}^{W}.$$
  Moreover, the norms are equivalent.
				
			 	 \end{theorem}
			 	 Since the converse of Theorem  \ref{tam3} is not true in general, we  have to  assume the hypothesis $M^P[f] <\infty$ in Theorem \ref{identities4}. In the case of the Hermite operator, since $\rho(x)=\frac{1}{1+|x|}\in L^\infty(\R^n)$, that hypothesis is not necessary (see Remark \ref{excep}) and the result  holds for $0<\alpha<2$. 			 	 
			 	  To prove Theorem  \ref{identities4} we need to introduce new {spaces} of functions $\Lambda_\alpha^{\tilde{P}}$, see Section \ref{secpoisson}, defined via the classical Poisson semigroup, that are more general than the ones defined by Stein in \cite{Stein} and we will compare {them} with the spaces $\Lambda_\alpha^P$.

	In \cite{BonHar}, the authors proved that, in the case $0<\alpha <1$, the space 		 $\Lambda^\alpha_\LL$ is isometric to the space $BMO^\alpha_\LL $  defined as the set of locally integrable functions such that, for every ball $B=B(x,R), R>0 $
\begin{eqnarray*} \int_B |f-f_B| &\le& C |B|^{1+\alpha/n}, \hbox{ with } f_B= \frac1{|B|} \int_B f \\
\hbox{ and }  \int_B |f| &\le& C|B|^{1+\alpha/n}, \hbox{  if   } 
R \ge \rho(x).
\end{eqnarray*}
Hence, our Theorems \ref{identities} and  \ref{identities4} can be viewed as a sort  of Carleson condition characterizations of the space $BMO^\alpha_\LL$. In the case of the Poisson semigroup, a complete Carleson characterization has been given in \cite{MSTZ} for a more restricted class of functions.
			
The organization of the paper is the following. In Section \ref{generalfacts} we collect  technical results about the heat kernel of the Schr\"odinger 	operators. Also, we analyze the spaces defined by using the heat kernel. We end the section by proving the natural growth at infinity of this class of functions (Proposition \ref{tam}). In Section 	\ref{comparativa}, we introduce an auxiliary space of functions defined by using the classical heat Gauss semigroup (Definition \ref{defclasica}). These spaces are characterized point-wise and also are compared with the classes defined through the heat semigroup associated to $\LL$. These two facts allow us to prove Theorem \ref{identities}. In Section \ref{pruebaaplicaciones}  	we prove  Theorems \ref{Schau}--\ref{multiplicador}, related with applications. Finally, 	Section \ref{secpoisson} is devoted to the proofs related with the $ \Lambda_{\alpha}^{P}$ spaces.		 	 

Along this paper, we will use the variable constant convention, in which $C$ denotes a
constant that may not be the same in each appearance. The constant will be written with
subindexes if we need to emphasize the dependence on some parameters.

					 \section{Properties of Heat Lipschitz spaces associated to $\LL$.}\label{generalfacts}
					 \subsection{Technical  results}
					 
					The nonnegative potential  $V$ is assumed to satisfy the following 
					reverse H\"older condition:
					  \begin{equation}\label{reverseHolder}
		\left( \frac{1}{|B|}\int_B V(y)^qdy\right)^{1/q}\le \frac{C}{|B|}\int_B V(y)dy, \hbox{ with an exponent } q>n/2,
		\end{equation}
		for every ball $B$.	
		Consider the critical radius function defined by 			
					 \begin{equation}\label{critical}\rho(x)=\sup\Big\{r>0: \: \frac{1}{r^{n-2}} \int_{B(x,r)}V(y)dy\le 1\Big\}.\end{equation}

									 Let $W_y(x,z)$ be the integral kernel of the semigroup of $e^{-y\LL}$ generated by $-\LL$. That is, for $f$  satisfying a {\bf heat size condition}
					 $$
					 e^{-y\LL}f(x)=\int_{\R^n}W_y(x,z)f(z)dz, \:\: x\in\R^n.
					 $$ 
					 It is  known 	 (see \cite{DZ1,Kurata}) that the integral kernel $W_\zeta(x,y)$ of the extension of $e^{-y\LL} $ to the holomorphic semigroup $\{e^{-\zeta \LL}\}_{\zeta \in \triangle_{\pi/4}} $ satisfies  
					 \begin{equation}\label{cotanucleo}
					|W_\zeta(x,z)|\le C_N\frac{e^{-\frac{|x-z|^2}{c \Re\zeta}}}{ (\Re\zeta)^{n/2}}\left(1+\frac{\sqrt{\Re \zeta }}{\rho(z)}+\frac{\sqrt{\Re \zeta}}{\rho(x)} \right)^{-N}, \:\:\; x,z\in\R^n,
					 \end{equation}
					for $N>0$ arbitrary.
					
					 \begin{lemma}\label{Cauchy}
					 	Let $k\ge 1$. There exist constants  $c, C_k>0$ such that, for every $M>0$,
					 	$$
					 	|\partial_y^kW_y(x,z)|\le C_k \frac{e^{-\frac{|x-z|^2}{cy}}}{y^{k+n/2}}\left(1+\frac{\sqrt{y}}{\rho(z)}+\frac{\sqrt{y}}{\rho(x)} \right)^{-M}.
					 	$$
					 		 \end{lemma}
							 The case $k=1$ of  this Lemma can be found  in \cite[Formula (2.7)]{DGMTZ} and \cite{DZ}.
							\begin{proof}
					 	By  Cauchy's integral formula  and \eqref{cotanucleo} we have 
					 	\begin{align*}| \partial_y^k W_y(x,z)| &= k! \Big|\frac{1}{2\pi i} \int_{|\zeta-y|=y/10} \frac{W_\zeta (x,z) }{(\zeta-y)^{k+1}} d\zeta\Big|  \le C_k \frac1{y^{k+n/2}} \Big( 1+ \frac{\sqrt{y}}{\rho(x)} + \frac{\sqrt{y}}{\rho(z)} \Big)^{-N} e ^{  -\frac{|x-z|^2}{cy}}.
					 	\end{align*}
					 \end{proof}
		\begin{remark}\label{consecuencia}
		A consequence of the last Lemma is that $\displaystyle \int_{\mathbb{R}^n} {|}\partial_y^kW_y(x,z) {|} dz \le \frac{C}{y^k}.$		
		\end{remark}

			\subsection{Controlled growth at infinity. Proof of Theorem \ref{tam2} }
	
	\
			
We shall denote by $\tilde{W}_y$ the Gauss kernel, in other words, the kernel of the classical heat semigroup $e^{y\Delta}$. The following Lemma is inspired in   \cite{Garrigos}, we sketch here the proof for completeness. 		\begin{lemma}\label{claim} Let $f$ be a measurable function such that there exists  $y_0>0$  for which $\int_{\mathbb{R}^n} e^{-\frac{|x|^2}{y_0}} f(x) dx < \infty$.  Then, $\lim_{y\rightarrow  0} W_yf(x) = f(x)$, a.e. $x\in\R^n$.
			\end{lemma}  
			\begin{proof} 
				Let $|x|\le A, \, A \in \mathbb{N}.$ Given a function $f$ we split
				$$f= f \chi_{\{|z|\le 2A\}} + f \chi_{\{|z|> 2A\}}  =f_1+f_2.$$ 
				Observe that, for $|z|>2A$,  $|x-z| \ge \frac{|z|}{2}$. Hence, by using (\ref{cotanucleo}), we get   for $y<y_0/(8c)$,   
		$$W_y(x,z)\le C\frac{e^{-\frac{|x-z|^2}{cy}}}{y^{n/2}} \le C\frac{e^{-\frac{|z|^2}{4cy}}}{y^{n/2}}  \le C\frac{e^{-\frac{A^2}{2cy}} e^{-\frac{|z|^2}{y_0}}}{y^{n/2}}.$$	
		Hence 
		$$|W_y f_2(x) | \le C	y^{-n/2} e^{-\frac{A^2}{2cy}} \int_{\mathbb{R}^n} |f(z)| e^{-\frac{|z|^2}{y_0}} dz  \rightarrow 0,  \, \hbox{ as } y \rightarrow 0.$$
			
			On the other hand, a function $\omega$ is said to be rapidly decaying if for every $N>0$ there exists a constant $C_N$ with $|\omega(x)| \le C_N(1+|x|)^{-N}.$   For a  rapidly decaying function $\omega$, we shall denote $w_y(x)=y^{-n/2}w(y^{-1/2}x)$, $y>0$, $x\in\R^n$.
 It is known, see \cite[Proposition 2.16]{DZ1}, that there exists a nonnegative rapidly decaying function $\omega$ such that  
			\begin{equation}\label{difcalor}			
		|W_y(x,z) - \tilde{W}_y(x-z)| \le C \Big( \frac{\sqrt{y}}{\rho(x)}\Big)^{2-n/q} \omega_y(x-z),  \hbox{  for  } \sqrt{y} \le  \rho(x).	\end{equation}
		Hence, for  $\sqrt{y} \le  \rho(x)$ ,
				$$|W_yf_1(x) - \tilde{W}_yf_1(x) | \le C \Big( \frac{\sqrt{y}}{\rho(x)}  \Big)^{2-n/q} \omega_y\star f_1(x)  .$$
				By the standard point-wise convergence for $L^1$-functions we have $$\lim_{y\rightarrow 0} \tilde{W}_y f_1(x) = f_1(x),   \text{and} \,\lim_{y\rightarrow 0} \omega_y\star f_1(x) = f_1(x) \text{ a.e }\, x\in\R^n.$$
				Therefore we get $$\lim_{y\rightarrow 0} W_y f_1(x) = f_1(x), \, \text{ a.e }\, x\in\R^n $$
	\end{proof}
\begin{proposition}\label{subir el $k$}
		Let $\alpha >0$, $k= [\alpha/2]+1$ and $f$ be a function satisfying the heat size condition. Then, $\| \partial_ y^{k} W_y f\|_ {L^\infty(\mathbb{R}^n)} \le C_\alpha y^{-k+ \alpha/2}$ if, and only if,   for $m \ge k$, $\| \partial_ y^{m} W_y f\|_ {L^\infty(\mathbb{R}^n)} \le C_m y^{-m+ \alpha/2}$. Moreover, for each $m$,  $C_m$ and $C_\alpha$ are comparable.
	\end{proposition}
	
	\begin{proof}  Let $m \ge  [\alpha/2]+1=k$. By the semigroup property  and Remark \ref{consecuencia} we have  
		\begin{eqnarray*} \Big| \partial_ y^{m}  W_yf(x) \Big|=  C \Big| \partial_y^{m-k}W_{y/2}(\partial_u^{k}W_{u}f(x)\big|_{u=y/2}) \Big|
			\le  C_\alpha'\frac1{y^{m-k}} y^{-k +\alpha/2} = C_m y^{-m+\alpha/2}. 
		\end{eqnarray*}
	
		 For the converse, the fact    $| \partial_ y^{\ell} W_y f(x) | \to 0 \text{ as } y\to\infty,$  allows us to integrate  on $y$ as many times as we need to  get $\| \partial_ y^{k}W_y f\|_ {L^\infty(\mathbb{R}^n)} \le C_\alpha\, y^{-k+ \alpha/2}.$
		\end{proof}

	To prove 
 Theorem \ref{tam2},   we need some lemmas and propositions that we present now.
  The following lemma can be found in \cite{ DZ,Shen}.
					 \begin{lemma}\label{Shenlemma}
					 	There exist constants $C>0$ and $k_0\ge1$ such that, for all $x,z\in\R^n$,
					 	\begin{align*}
					 	C^{-1}{\rho(x)}\left(1+\frac{|x-z|}{\rho(x)} \right)^{-k_0}\le{\rho(z)}\le		{C\rho(x)\left(1+\frac{|x-z|}{\rho(x)} \right)^{\frac{k_0}{1+k_0}}}.
					 	\end{align*}
					 	In particular, $\rho(x)\sim \rho(z)$ when $z\in B_r(x)$ and $r\le C\rho(x)$.
					 \end{lemma}
				
						\begin{lemma}\label{sizeheat}
							Let $f$ be a function such that $\rho(\cdot)^{-\alpha}f\in L^\infty(\R^n)$, for some $\alpha>0$. Then,  for every $\ell\in\N\cup\{0\}$ and $M>0$, $|\partial_y^\ell W_y f(x)|\le C_\ell \, M^\LL_\alpha[f] \, \frac{(\rho(x))^{\alpha}}{y^\ell} \left( 1+\frac{y^{1/2}}{\rho(x)}\right)^{-M}$, $x\in\R^n$, $y>0$.
						\end{lemma}
						\begin{proof}
							By Lemata \ref{Cauchy} and  \ref{Shenlemma}, if $\omega$ denotes a rapidly decaying function and  $w_y(x)=y^{-n/2}w(y^{-1/2}x)$, we have 							\begin{align*}
							&| W_y f(x)|\le 
\int_{\mathbb{R}^n} \omega_y(x-z) \Big(1+\frac{\sqrt{y}}{\rho(x)}\Big)^{-N-\alpha \lambda} |f(z)| \rho(z)^{-\alpha} \rho(x)^\alpha \Big(1+\frac{|x-z|}{\rho(x)}\Big)^{\alpha\lambda} dz\\	
& \lesssim  M^\LL_\alpha[f] \rho(x)^\alpha 
\Big(1+\frac{\sqrt{y}}{\rho(x)}\Big)^{-N} 
\int_{\mathbb{R}^n} \omega_y(x-z) 
\Big(1+\frac{\sqrt{y}}{\rho(x)}\Big)^{-\alpha \lambda} 
\Big(1+\frac{|x-z|}{\rho(x)}\Big)^{\alpha \lambda} dz\\
							&\lesssim    \,M^\LL_\alpha[f] \, \rho(x)^\alpha \left( 1+\frac{\sqrt{y}}{\rho(x)}\right)^{-N}.  
\end{align*}
							 For the derivatives, we proceed in the same way by using Lemma \ref{Cauchy}.
						\end{proof}

	The following Proposition is a direct consequence of Lemma \ref{sizeheat}. Moreover, it corresponds with the 			
 `` if '' part of Theorem \ref{tam2}.
				
				\begin{proposition} Let  $\LL$ be a Schr\"odinger operator with a reverse H\"older class potential and associated function $\rho$.  Let { $\alpha>0$ } and $f$ a measurable function such that    $\rho(\cdot)^{-\alpha}f\in L^\infty(\R^n)$, then $f$ satisfies a {\bf heat size condition} for $\LL$.
				\end{proposition}

\begin{lemma}\label{cambiaryrho}
					Let $\alpha>0$ and $k=[\alpha/2]+1$. Assume that $f$ satisfies the {\bf heat size condition } and (\ref{derivada}), then for every $j,m\in\N\cup \{0\}$ such that $\frac{m}{2}+j\ge k$,  there exists a $C_{m,j}>0$ such that
					$$
					\left\| \frac{\partial_y^j W_yf}{\rho(\cdot)^m}\right\|_{\infty}\le C_m	S^W_\alpha[f] y^{-(\frac{m}{2}+j)+\alpha/2}.
					$$
				\end{lemma}
				\begin{proof}  For  $\ell\ge k$, by the semigroup property and Lemma \ref{Cauchy} we get that
					\begin{align*}
					\left|\frac{\partial_y^\ell W_yf(x)}{\rho(x)^m}\right|&=\left|\frac{C_\ell }{\rho(x)^m}\int_{\R^n} \partial_v^{\ell-k}W_{v}(x,z)|_{v=y/2}\partial_u^k W_uf(z)|_{u=y/2}dz\right|\\
					&\le \frac{C_\ell \|\partial_u^k W_uf|_{u=y/2}\|_\infty}{\rho(x)^m}\int_{\R^n} \frac{e^{-\frac{|x-z|^2}{cy}}}{y^{n/2+\ell-k}}\left(\frac{\rho(x)}{y^{1/2}} \right)^{m}dz\\
					&\le  C_\ell	S^W_\alpha[f]y^{-(\frac{m}{2}+\ell)+\alpha/2}, \:\; x\in\R^n.
					\end{align*}
						If  $j<k$, since the $y-$derivatives of $W_yf(x)$ tend to zero as $y\to \infty$, we  integrate $\ell-j$ times the previous estimate  and we get the result.

				\end{proof}

The following Proposition corresponds with the ``only if '' part of Theorem \ref{tam2}.
				
				\begin{proposition}\label{tam} 
					Let {$\alpha>0$} and $f$ be a function satisfying the {\bf heat size condition} for $\LL$ and (\ref{derivada}). Then  $\rho(\cdot)^{-\alpha}f\in L^\infty(\R^n)$.
				\end{proposition}

				\begin{proof}

					By using Lemma \ref{claim}, for a.e. $x$,  we have 
					\begin{align*}
					|f(x)| &\le  \sup_{0< y < \rho(x)^2} |W_yf(x) | \\ &\le 
					\sup_{0< y < \rho(x)^2} |W_yf(x)-W_{\rho(x)^2}f(x) | + 
					|W_{\rho(x)^2}f(x)| \\ &= I+II.
					\end{align*}	
We shall estimate  $I$. 	Let $k = [\alpha/2]+1$. If $\alpha$ is not even, by Lemma \ref{cambiaryrho} with $j=1$ and $m =2(k-1)$  we have that
					\begin{align*} 
					I &\le   \rho(x)^{2(k-1)} \sup_{0< y < \rho(x)^2} \int_y^{\rho(x)^2} \left|\frac{\partial_z W_ z f(x)}{\rho(x)^{2(k-1)}}\right| dz   \le C
					S^W_\alpha[f] \rho(x)^{2(k-1)} \sup_{0< y < \rho(x)^2} \int_y^{\rho(x)^2}  z^{-k+\alpha/2} dz  \\ &\le 
					CS^W_\alpha[f]\rho(x)^{2(k-1)}   \sup_{0< y < \rho(x)^2} ((\rho(x)^2)^{-(k-1)+\alpha/2}- y^{-(k-1)+\alpha/2}) \le CS^W_\alpha[f] \rho(x)^{\alpha}.
					\end{align*}
					When $\alpha$ is even, we  write 
					\begin{align*}
					I&= \sup_{0< y < \rho(x)^2}\left| \int_y^{\rho(x)^2}{\partial_z W_ z f(x)}dz\right|\\&=\sup_{0< y < \rho(x)^2}\left| \int_y^{\rho(x)^2}\left(-\int_{z}^{\rho(x)^2}\partial_u^2 W_ u f(x)du+\partial_v W_v f(x)|_{v=\rho(x)^2}\right)dz\right|.
					\end{align*}		
					By Lemma \ref{cambiaryrho} with $j=2$ and $m= 2(k-2)$, since $k= \alpha/2+1$,  we get
					\begin{align*}
					\Big| &\int_y^{\rho(x)^2}\int_{z}^{\rho(x)^2}\partial_u^2 W_ u f(x)dudz\Big|=\rho(x)^{2(k-2)}\left| \int_y^{\rho(x)^2}\int_{z}^{\rho(x)^2}\frac{\partial_u^2 W_ u f(x)}{\rho(x)^{2(k-2)}}dudz\right|\\
					&\le CS^W_\alpha[f]\rho(x)^{\alpha-2} \int_y^{\rho(x)^2}\int_{z}^{\rho(x)^2}u^{-1}dudz= CS^W_\alpha[f] \rho(x)^{\alpha-2} \int_y^{\rho(x)^2}(\log(\rho(x)^2)-\log z)dz\\
					&= CS^W_\alpha[f] \rho(x)^{\alpha-2}\big[\log(\rho(x)^2)(\rho(x)^2-y)-(\rho(x)^2\log(\rho(x)^2)-\rho(x)^2-y\log y+y)\big]\\
					&=C S^W_\alpha[f] \rho(x)^{\alpha-2}\big[y\log \big(\frac{y}{\rho(x)^2}\big)+\rho(x)^2-y\big]\le CS^W_\alpha[f] \rho(x)^{\alpha}.
					\end{align*}	
					For the second summand of $I$, Lemma \ref{cambiaryrho}, with $j=1$ and $m= 2(k-1)$ applies, 	 so 
					\begin{align*} 
					\sup_{0< y < \rho(x)^2}	(\rho(x)^2-y)|\partial_v W_v f(x)|_{v=\rho(x)^2}|&=	\sup_{0< y < \rho(x)^2}	(\rho(x)^2-y)\rho(x)^{2(k-1)}\frac{|\partial_v W_v f(x)|_{v=\rho(x)^2}|}{\rho(x)^{2(k-1)}}\\&\le 	C S^W_\alpha[f]\sup_{0< y < \rho(x)^2}	(\rho(x)^2-y)\rho(x)^{\alpha} (\rho(x)^2)^{-1}\\
					&\le C S^W_\alpha[f]\rho(x)^{\alpha}.	
					\end{align*}
Regarding $II$, by using Lemma \ref{cambiaryrho} with $j=0$ and $m =2k$ we have \begin{align*}
					II=|W_{\rho(x)^2}f(x)|= \left|\frac{W_{\rho(x)^2}f(x)}{\rho(x)^{2k}} \right|\rho(x)^{2k}\le C S^W_\alpha[f] (\rho(x)^2)^{-k+\alpha/2}\rho(x)^{2k}=C S^W_\alpha[f]\rho(x)^\alpha.
					\end{align*}					
					
				\end{proof}

				\section{Proof of  Theorem  \ref{identities}}\label{comparativa}

\subsection{Some remarks about  the classical Lipschitz spaces}

	In this subsection we define a class of Lipschitz spaces associated to Laplace operator. It will be an auxiliary class for our results about the spaces adapted to the Schr\"odinger operator.  With respect to the classical definitions, see \cite{Stein}, \cite{Taibleson}, the main and crucial difference is that the functions don't need to be bounded.

\begin{definition}\label{defclasica} Let $\alpha >0.$ We define the spaces $\Lambda_{\alpha/2}^{\tilde{W}}$  as
\begin{align*}\Lambda_{\alpha/2}^{\tilde{W}}=&\Big\{f
	:\;  (1+|\cdot|)^{-\alpha} f \in  L^\infty(\R^n) \hbox{ and  }
	 \left\|\partial_y^k	\tilde{W}_y f \right\|_{L^\infty(\mathbb{R}^{n})}\leq C_\alpha y^{-k+\alpha/2},\:  k=[\alpha/2]+1  \Big\}.\end{align*} 
	 Parallel to the linear spaces  $\Lambda_{\alpha/2}^{W}$, we can endow this class with the norm $$\|f\|_{\Lambda_{\alpha/2}^{\tilde{W}}} := \tilde{M}_\alpha [f] + \tilde{S}_\alpha[f], $$
with $\tilde{M}_\alpha [f]=\|(1+|\cdot|)^{-\alpha} f(\cdot)\|_\infty$ and $\tilde{S}_\alpha[f]$ being the infimum of the constants $C_\alpha$ appearing above.
\end{definition} 

\begin{remark}\label{goingtoinfinity} Let $f$ be a function such that  $\tilde{M}_\alpha[f] < \infty$. Then, for every $\ell\in \N\cup\{0\}$, $\partial_y^\ell	\tilde{W}_y f$ is well defined. Observe that 
 \begin{eqnarray*}
\int_{\mathbb{R}^n} \frac{e^{-\frac{|x-z|^2}{cy}}}{y^{n/2}} {|f(z)|} dz \le {C} \int_{\mathbb{R}^n} \frac{e^{-\frac{|x-z|^2}{cy}}}{y^{n/2}}(1+|z|)^\alpha dz. 
\end{eqnarray*}
If $|z|< 2|x|$,  the  last integral is convergent and bounded by $C(1+y^{\alpha/2}+|x|^{\alpha})$. If $|z|> 2|x|$ then  the above integral is less than \begin{eqnarray*}
\int_{\mathbb{R}^n} \frac{e^{-\frac{|z|^2}{cy}}}{y^{n/2}}(1+|z|)^{\alpha}  dz\le C(1+y^{\alpha/2}). 
\end{eqnarray*} 
The same arguments can be used for the derivatives $\partial_y^\ell	\tilde{W}_y f$, $\ell\in\N$.

Moreover, if $m/2+\ell \ge [\alpha/2]+1$, then $\lim_{y \to \infty }\partial_{x_i}^m\partial_y^\ell\tilde{W}_yf(x) =0$, for every $x\in\R^n.$
Indeed, observe that  
\begin{eqnarray*}
\Big|\partial_{x_i}^m\partial_y^\ell\tilde{W}_yf(x) \Big|\le C \int_{\mathbb{R}^n} \frac{e^{-\frac{|x-z|^2}{cy}}|f(z)|}{y^{n/2+m/2+\ell}} dz \le C\int_{\mathbb{R}^n} \frac{e^{-\frac{|x-z|^2}{cy}}}{y^{n/2+m/2+\ell}}(1+|z|)^\alpha dz. 
\end{eqnarray*} 
If $|z|< 2|x|$, the last integral is less than $C(1+y^{\alpha/2}+|x|^\alpha) y^{-m/2-\ell}$. In  the case $|z|> 2|x|$ the integral is less than $C(1+y^{\alpha/2}) y^{-m/2-\ell}.$

\end{remark}

The following Lemma is parallel to Lemma \ref{claim} and follows from the ideas in   \cite{Garrigos}. We sketch  the proof for completeness. 
	
\begin{lemma}\label{claimclasico} Let $f$ be a measurable function such that, for every $y>0$,  $\int_{\mathbb{R}^n} e^{-\frac{|x|^2}{y}} |f(x)| dx < \infty$.  Then, $\lim_{y\rightarrow  0} \tilde{W}_y f(x) = f(x),$ a.e. $x\in\R^n$. Moreover, $\tilde{W}_y f(x)$ belongs to $C^\infty((0,\infty) \times \R^n)$.
	
	\end{lemma}  
	\begin{proof}
	Since
		$$
		\left|\int_{\R^n}\partial_y^\ell \tilde{W}_y(x,y)f(z)dz\right|\le \frac{C}{y^\ell}\int_{\R^n}\frac{e^{-\frac{|x-z|^2}{cy}}}{y^{n/2}}|f(z)|dz,
		$$
		we can {differentiate} $\tilde{W}_yf$ with respect to $y.$
		
		On the other hand, observe that
		$$\left|\int_{\R^n}\partial_{x_i}\tilde{W}_y(x-z)f(z)dz\right|\le \frac{C}{y^{\frac{n+1}{2}}} \int_{\mathbb{R}^n} e^{-\frac{|x-z|^2}{cy}} |f(z)|dz, \: y>0.$$
		Given  $x\in \mathbb{R}^n$ with $|x| < R$, {for some $R>0$}, we have 
		$$e^{-\frac{|x-z|^2}{cy}} |f(z)| \le C \Big(  \chi_{|z|> 2 R} \, e^{{-\frac{|z|^2}{Cy}}}+  \chi_{|z|< 2 R}\,  \Big) |f(z)|,
		$$
 so we can  {differentiate}  $\tilde{W}_yf$ with respect to $x_i$, $i=1,\dots,n$.
	\end{proof}
	
\begin{proposition}\label{subir el $k$cla}
		Let {$\alpha>0$}. A function   $f\in \Lambda_{\alpha/2}^{\tilde{W}}$if, and only if,   for  all $m \ge [\alpha/2]+1$, we have $\| \partial_ y^{m} \tilde{W}_y f\|_ {L^\infty(\mathbb{R}^n)} \le C_m y^{-m+ \alpha/2}$ and  $\tilde{M}_\alpha[f]<\infty.$
	\end{proposition}
	The proof of this Proposition is parallel to the proof of Proposition \ref{subir el $k$}, we leave the details to the reader.

	\begin{lemma}\label{derivX}
		Let $\alpha>0$ and $k=[\alpha/2]+1$. If $f\in \Lambda_{\alpha/2}^{\tilde{W}}$,  then for every $j,m\in\N\cup \{0\}$ such that $\frac{m}{2}+j\ge k$,  there exists a $C_{m,j}>0$ such that
		$$
		\left\| {\partial_{x_i}^m\partial_y^j \tilde{W}_yf}\right\|_{\infty}\le C_{m,j}	\tilde{S}_\alpha[f]\, y^{-(m/2+j)+\alpha/2}, \text{ for every } i=1\dots,n.
		$$
		Moreover,  for each ${j,m},$ the constant  $C_{m,j}$ is comparable to the constant $C_\alpha$ in Definition \ref{defclasica}.	
		\end{lemma}
		
	\begin{proof}If $j\ge k$, by the semigroup property we get that
		\begin{align*}
		\left|\partial_{x_i}^m\partial_y^j\tilde{W}_yf(x)\right|&=C  \left|\int_{\R^n} \partial_{x_i}^m\partial_v^{j-k}\tilde{W}_{v}(x-z)|_{v=y/2}\partial_u^k \tilde{W}_uf(z)|_{u=y/2}dz\right|\\
		&\le  \frac{C_{m,j} \|\partial_u^k \tilde{W}_uf|_{u=y/2}\|_\infty}{y^{\frac{m}{2}+j-k}} \int_{\R^n}\frac{e^{-\frac{|x-z|^2}{cy}}}{y^{n/2}}dz\\
		&\le C_{m,j}	\tilde{S}_\alpha[f] \, y^{-(\frac{m}{2}+j)+\alpha/2}, \:\; x\in\R^n.
		\end{align*}
 If $j<k$, by proceeding as before we get that $	\left|\partial_{x_i}^m\partial_y^k\tilde{W}_yf(x)\right|\le  C	\tilde{S}_\alpha[f] y^{-(\frac{m}{2}+k)+\alpha/2}$, $ x\in\R^n$, and  we  get the result by integrating the previous estimate $k-j$ times, since  $|\partial_{x_i}^m\partial_y^\ell\tilde{W}_yf(x)|\to 0$ as $y\to \infty$ as far as $\frac{m}{2}+\ell \ge k$, see Remark \ref{goingtoinfinity}. 	
 
 	\end{proof}

			\begin{theorem}\label{nuevoStein}
				Let $0<\alpha < 2.$  Then  $f\in \Lambda_{\alpha/2}^{\tilde{W}}$ if, and only if   $$N_\alpha[f] := \sup_{|z|>0}\frac{\|f(\cdot+z)+f(\cdot-z)-2f(\cdot)\|_\infty}{|z|^\alpha}  <\infty \hbox { and  } \tilde{M}_\alpha [f]=  \|(1+|\cdot|)^{-\alpha} f\|_\infty<\infty.$$
				Moreover, $$\|f \|_{\Lambda_{\alpha/2}^{\tilde{W}} }  \sim N_\alpha[f] + \tilde{M}_\alpha [f]. $$
			\end{theorem}
		
		\begin{proof}
			Let $x\in\R^n$ and $f\in \Lambda_{\alpha/2}^{\tilde{W}}$. We can write, for every $y>0$, $z\in\R^n$,
				\begin{align*}
				|f(x+z)&+f(x-z)-2f(x)|\le |\tilde{W}_yf(x+z)-f(x+z)|+|\tilde{W}_yf(x-z)-f(x-z)|\\&\quad \quad+2|\tilde{W}_yf(x)-f(x)|+ |\tilde{W}_yf(x+z)-\tilde{W}_yf(x) +\tilde{W}_yf(x-z)-\tilde{W}_yf(x)|.
				\end{align*}
				By using Lemma \ref{claimclasico} we have that
				$$
				|\tilde{W}_yf(x)-f(x)|=\left|\int_0^y \partial_u \tilde{W}_uf(x)du \right|\le C\tilde{S}_\alpha[f]\int_0^y u^{-1+\alpha/2}du =C \tilde{S}_\alpha[f] y^{\alpha/2}.
				$$
		In a parallel way we handle the two first	summands.  Regarding the last sumand,  by using the chain rule and Lemma \ref{derivX} we have that \begin{align*}|\tilde{W}_yf(x+z)-&\tilde{W}_yf(x) +\tilde{W}_yf(x-z)-\tilde{W}_yf(x)|=
				\left|\int_{0}^1\partial_{\theta}(\tilde{W}_yf(x+\theta z)+\tilde{W}_yf(x-\theta z)) d\theta  \right|\\&=
				 \left|\int_0^1(\nabla_u \tilde{W}_yf(u)_{|_{u=x+\theta z}}\cdot z-\nabla_v \tilde{W}_yf(v)_{|_{v=x-\theta z}}\cdot z)d\theta\right|\\&=  \left|\int_0^1\int_{-1}^1\partial_\lambda\nabla_u \tilde{W}_yf(u)_{|_{u=x+\lambda\theta z}}\cdot z\,d\lambda d\theta\right|\\&=\left|\int_0^1\int_{-1}^1\nabla_u^2 \tilde{W}_yf(u)_{|_{u=x+\lambda\theta z}}z\cdot\theta z\,d\lambda d\theta\right|\\&\le C\tilde{S}_\alpha[f] \, y^{-1+\alpha/2}|z|^2, \:\: \end{align*}
				 	Thus, by choosing $y=|z|^2$ we get what we wanted. 
					
					For the converse, we assume that $ N_\alpha[f],\tilde{M}_\alpha [f]<\infty$. Since \newline $\displaystyle  \int_{\mathbb{R}^{n}} \partial_y \tilde{W}_y(z)f(x+z) dz=\int_{\mathbb{R}^{n}} \partial_y \tilde{W}_y(-z)f(x-z)dz=\int_{\mathbb{R}^{n}} \partial_y \tilde{W}_y(z)f(x-z)dz,$ 
				\newline and  $\displaystyle \int_{\R^n}\partial_y \tilde{W}_y(z)dz=0$ we have
				\begin{align*}
				|	\partial_y \tilde{W}_yf(x)| &= \left|\frac{1}{2}\int_{\mathbb{R}^{n}}\partial_y \tilde{W}_y(z)(f(x-z)+f(x+z)-2f(x)) dz\right| \\&\le C N_\alpha[f]\,  \int_{\mathbb{R}^n}\frac{e^{-\frac{|z|^2}{cy}}|z|^\alpha}{y^{\frac{n}{2}+1}}dz\le C N_\alpha[f]\,   y^{-1+\alpha/2}. 
				\end{align*}

		\end{proof}

The following Proposition shows that in the case $0<\alpha < 1$ we recover the classical Lipschitz condition.

		\begin{proposition}\label{Pola} Let $0<\alpha < 1$. If a function $f \in \Lambda^{\tilde{W}}_{\alpha/2}$ then $$\sup_{|z|>0} \frac{\|f(\cdot-z)-f(\cdot) \|_\infty}{|z|^\alpha} < \infty.$$ 

\end{proposition}

\begin{proof}  We assume that $f \in \Lambda^{\tilde{W}}_{\alpha/2}$ with $\|f\|_{\Lambda^{\tilde{W}}_{\alpha/2}} =1$. Let us take a representative of the function $f$.  We want to show that $|f(x+z) -f(x)| \le C|z|^\alpha, $ $x,z\in\R^n$.

Fix $x \in \mathbb{R}^n.$   Assume first that $|x|>1$. In the case  $|z| \ge |x|$, as $\tilde{M}_\alpha[f] < \infty$, we have that $|f(x+z) -f(x) |\le C(1+|x|+|z|)^\alpha \le C |z|^\alpha.$ In the case $|z|< |x|$, we choose a  nonnegative integer $k$ such that   $|x| \le  |2^k z | < 2|x|.$
We define $$g(t) = f(x+t) -f(x), \quad   t \in \mathbb{R}^n.$$

 By hypothesis and Theorem \ref{nuevoStein}, $$|g(t) - 2g(t/2)| =
|f(x+t)+f(x)-2f(x+ t/2)| \le C{|t|}^\alpha.$$
Similarly, $\displaystyle | 2^{j-1} g(t/2^{j-1}) - 2^j g(t/2^j)|\le C2^{j-1} \Big(\frac{|t|}{2^{j-1}}\Big)^\alpha$.
Therefore, adding up we have 
$$|g(t)- 2^k g(t/2^k)|\le C\sum_{j=1}^k 2^{j-1}\Big(\frac{|t|}{2^{j-1}}\Big)^\alpha.$$ 
Now we choose $t= 2^k z$. We have 
\begin{eqnarray*}|g(z) | &\le& \frac{|g(2^k z)|}{2^k} +  C\frac1{2^k}\sum_{j=1}^k 2^{j-1}\Big(\frac{2^k |z|}{2^{j-1}}\Big)^\alpha\\  &\le& C |z|^\alpha 2^{k(\alpha-1)}+ C2^{k(\alpha-1)} |z|^\alpha \sum_{j=0}^k 2^{j(1-\alpha)}  \le C |z|^\alpha.
\end{eqnarray*}
This implies that $|f(x+z)-f(x)| \le C|z|^\alpha.$

If $|x|<1 < |z|$  we can proceed as in the previous case $|x| < |z|$. If $|x|<1$ and $|z|<1$, we choose $k $ such that 
$1\le |2^k z| < 2$. We observe that in this case $|g(2^k z)|\le C$, therefore 
\begin{eqnarray*}|g(z) | &\le& C\frac{|g(2^k z)|}{2^k} +  C\frac1{2^k}\sum_{j=0}^k 2^{j-1}\Big(\frac{2^k |z|}{2^{j-1}}\Big)^\alpha\\  &\le& C  |z|+ 2^{k(\alpha-1)} |z|^\alpha \sum_{j=0}^k 2^{j(1-\alpha)}  \le C |z|^\alpha.
\end{eqnarray*}
Observe that we have used in an essencial way that $0< \alpha <1.$

\end{proof}

\begin{remark}\label{espaciopola}
	Observe that Lemma \ref{Shenlemma} for $x=0$, i.e.
	\begin{equation}\label{shen0}
		\rho(z)\le C\rho(0)\left(1+\frac{|z|}{\rho(0)}\right)^{\lambda}, \text{ for some } 0<\lambda<1,
	\end{equation}
	implies that if $\rho(\cdot)^{-\alpha}f\in L^\infty(\R^n)$, then $(1+|\cdot|)^{-\alpha}f\in L^\infty(\R^n)$. Therefore, for $0<\alpha<1$, Theorem \ref{nuevoStein} and Proposition \ref{Pola} imply that  $\Lambda^\alpha_\LL$ coincides with the space  {introduced in \cite{BonHar}, see \eqref{polabruno}}.
	\end{remark}

	\begin{proposition}\label{Steinderivadax}
	  	Let $1< \alpha<2$, $f\in  \Lambda_{\alpha/2}^{\tilde{W}}$ and assume that for a certain $\rho$ associated to a  Schr\"odinger operator $\LL$, we have    $\rho(\cdot)^{-\alpha}f\in L^\infty(\R^n)$. Then,  for every $i=1,\dots, n$, $\partial_{x_i}f\in \Lambda_{\frac{\alpha-1}{2}}^{\tilde{W}}$ and $\rho(\cdot)^{-(\alpha-1)}\partial_{x_i}f\in L^\infty(\R^n)$. Moreover, $$\| \partial_{x_i} f\|_{\Lambda_{\frac{\alpha-1}{2}}^{\tilde{W}}} \le C\Big(\tilde{S}_\alpha[f] + M_\alpha^{\LL}[f]\Big).$$
	  \end{proposition}
	  \begin{proof}

	  	We first prove that $\partial_{x_i}f$ exists. 	By Lemma \ref{derivX} we have that $\|\partial_y\partial_{x_i}\tilde{W}_yf\|_\infty\le C\, \tilde{S}_\alpha[f]\,y^{-3/2+\alpha/2}=C\, \tilde{S}_\alpha[f]\,y^{-1+\frac{\alpha-1}{2}}.$
	  	For every $x\in\R^n$ we can write
	  	$$
	  	\partial_{x_i}\tilde{W}_yf(x)=-\int_y^1\partial_u\partial_{x_i}\tilde{W}_uf(x)du+\partial_{x_i}\tilde{W}_yf(x)|_{y=1}.
	  	$$
	  	Therefore, for every $0<y_1<y_2<1$ we have 
	  	\begin{align*}
	  	|\partial_{x_i}\tilde{W}_{y_2}f(x)-\partial_{x_i}\tilde{W}_{y_1}f(x)|&=\left|\int_{y_1}^{y_2}\partial_u\partial_{x_i}\tilde{W}_uf(x)du \right|\\&\le C|y_2^{\frac{\alpha-1}{2}}-y_1^{\frac{\alpha-1}{2}}|\le C |y_2-y_1|^{\frac{\alpha-1}{2}}.
	  	\end{align*}
	  	This means that $\{\partial_{x_i}\tilde{W}_{y}f\}_{y>0}$ is a Cauchy sequence in the $L^\infty$ norm (as $y\to 0$). In addition, as $\tilde{W}_yf\to f$ as $y\to 0$ we get that $\partial_{x_i}\tilde{W}_{y}f$ converges uniformly to $\partial_{x_i}f$.
	  	
	  	On the other hand, since $\rho(\cdot)^{-\alpha}f\in L^\infty(\R^n)$,  {integration by parts} and \eqref{shen0} give \newline 
	$\displaystyle\Big|\int_{\R^n}e^{-\frac{|z|^2}{y}}\partial_{z_i}f(z)dz\Big|=\Big|\int_{\R^n}e^{-\frac{|z|^2}{y}}\frac{2z_i}{y}f(z)dz\Big|\le C\int_{\R^n}\frac{e^{-\frac{|z|^2}{cy}}}{y^{1/2}}|f(z)|dz <\infty$, for every $y>0$.
	  Moreover, since $\tilde{W}_yf$ is a convolution, by Remark \ref{goingtoinfinity}  we have $|\partial_y\tilde{W}_y(\partial_{x_i}f)(x)|=|\partial_y\partial_{x_i}\tilde{W}_yf(x)|\le C\tilde{S}_\alpha[f]\, y^{-(3/2)+\alpha/2}=C\,\tilde{S}_\alpha[f]\, y^{-1+(\alpha-1)/2}$.
	  	
	  	Let us see the size condition for the derivative. By proceeding as in the proof of  Proposition \ref{tam}, we have
	  	\begin{align*}
	  	\frac{|\partial_{x_i}f(x)|}{\rho(x)^{\alpha-1}} &\le  \frac1{\rho(x)^{\alpha-1}}\sup_{0< y < \rho(x)^2} |\tilde{W}_y(\partial_{x_i}f)(x) | \\ &\le 
	  	\frac1{\rho(x)^{\alpha-1}}\sup_{0< y < \rho(x)^2} |(\tilde{W}_y(\partial_{x_i}f)(x)-\tilde{W}_{\rho(x)^2}(\partial_{x_i}f)(x) | + 
	  	\frac1{\rho(x)^{\alpha-1}} |\tilde{W}_{\rho(x)^2}(\partial_{x_i}f)(x)| \\ &= I+II.
	  	\end{align*}
	  	\begin{align*} 
	  	I &\le   \frac1{\rho(x)^{\alpha-1}} \sup_{0< y < \rho(x)^2} \int_y^{\rho(x)^2} |\partial_z \tilde{W}_z (\partial_{x_i}f)(x)| dz   \le 
	  	C\frac{\tilde{S}_\alpha[ f]}{\rho(x)^{\alpha-1}} \sup_{0< y < \rho(x)^2} \int_y^{\rho(x)^2}  z^{-1+\frac{\alpha-1}{2}} dz  \\ &\le 
	  	C \frac{\tilde{S}_\alpha[ f]}{\rho(x)^{\alpha-1}} \sup_{0< y < \rho(x)^2} (\rho(x)^{\alpha-1}- y^{\frac{\alpha-1}{2}}) \le C\tilde{S}_\alpha[ f].
	  	\end{align*}
	  	On the other hand, integration by parts and Lemma \ref{Shenlemma} give
	  	\begin{align*}
	  	II&=\frac1{\rho(x)^{\alpha-1}} \left| \int_{\R^n}\partial_{z_i}\tilde{W}_{\rho(x)^2}(x-z)f(z)dz\right|\le \frac{C\,  M_\alpha^{\LL}[f]}{\rho(x)^{\alpha-1}}\int_{\R^n} \frac{e^{-\frac{|x-z|^2}{c\rho(x)^2}}}{\rho(x)^{n+1}}\rho(z)^\alpha dz\\
	  	&\le		\frac{C\,  M_\alpha^{\LL}[f]}{\rho(x)^{\alpha-1}}\int_{\R^n} \frac{e^{-\frac{|x-z|^2}{c\rho(x)^2}}}{\rho(x)^{n+1}}\rho(x)^\alpha \left( 1+\frac{|x-z|}{\rho(x)}\right)^{\lambda \alpha} dz.
		\end{align*}
		Performing the change of variable $\tilde{z} = (x-z)/\rho(x)$ we get $II\le  C\,  M_\alpha^{\LL}[f].$

	  	Finally, \eqref{shen0} allows us to conclude that  $\tilde{M}_{\alpha-1}[\partial_{x_i} f] < \infty$ and hence $\partial_{x_i}f\in \Lambda_{\frac{\alpha-1}{2}}^{\tilde{W}}$.

	  \end{proof}

			\subsection{Comparison of Lipschitz spaces. $\Delta$  versus   $\LL$.}
	Along this section we shall need the following result that can be found in \cite{DZ}, \cite{Shen}.	 Recall the definition of a rapidly decaying nonnegative function from the proof of Lemma \ref{claim}.							
						\begin{lemma}\label{estV} Let $\omega$ be a rapidly decaying nonnegative function and consider $\omega_y(x)=y^{-n/2}\omega(y^{-1/2}x)$, $y>0$, $x\in\R^n$. There exists a constant $C>0$ such that
					 	$$
					 	\int_{\R^n} V(z)\omega_y(x-z)dz\le C \frac{1}{y}\left(\frac{ y^{1/2}}{\rho(x)}\right)^{2-\frac{n}{q}},		\text{ whenever } y\le \rho(x)^2.				$$
\end{lemma}		
				
			\begin{theorem}\label{comparacion}
				Let $0< \alpha \le 2-n/q$, and a function $f$ such that   $\rho(\cdot)^{-\alpha}f(\cdot) \in L^\infty(\R^n)$. Then, $\|\partial_t \tilde{W}_tf-\partial_tW_tf\|_\infty\le C\, M_\alpha^{\LL}[f]\, t^{-1+\alpha/2}$. 			\end{theorem}
			\begin{proof} Let $t>0$ and $x\in\R^n$. The existence of the derivatives $\partial_t \tilde{W}_tf(x)$ and $\partial_t W_tf(x)$ follows from  Lemma  \ref{sizeheat} and Remark  \ref{goingtoinfinity}. 				We analyze first the case  $t\le \rho(x)^2$. As a consequence of the Kato-Trotter formula,
				$$\tilde{W}_t(x-y)-W_t(x,y) =\int_0^t \int_{\mathbb{R}^n}\tilde{W}_{t-s}(x-z)V(z)W_s(z,y) dzds,$$
				see {\cite{DZ}},  we have the following identity:
				\begin{eqnarray*}
				\nonumber \partial_t (\tilde{W}_tf -  W_tf) &=&  \int_0^{t/2} \frac{\partial }{ \partial t} \tilde{W}_{t-s} V W_sf ds +
				\int_{t/2}^t  \tilde{W}_{t-s} V \frac{\partial}{\partial s} W_sf ds  +  \tilde{W}_{t/2}V W_{t/2}f \\
				&=& A+B+E.
				\end{eqnarray*}	 
				Assume $t\le \rho(x)^2$, then by Lemmata   \ref{Cauchy},	\ref{Shenlemma} and 	\ref{estV}, if $\omega$ denotes a rapidly decaying function and  $\omega_t(x)=t^{-n/2}\omega(t^{-1/2}x)$,  we have	
				\begin{align*}
				|A| &\lesssim  \int_0^{t/2} \int_{\mathbb{R}^n}\int_{\mathbb{R}^n} t^{-1} \omega_t(x-z)V(z) \omega_s(z-y) |f(y) | \rho(y)^{-\alpha}\rho(x)^\alpha \Big(1+\frac{|x-y|}{\rho(x)}\Big)^{\alpha\lambda}\,dy\,dz\,ds\\
				&\lesssim M_\alpha^{\mathcal{L}}[f]\rho(x)^\alpha t^{-1}\\
				&\quad \times 
				\int_0^{t/2} \int_{\mathbb{R}^n}\int_{\mathbb{R}^n} \omega_t(x-z)V(z) \omega_s(z-y) \Big(1+\frac{|x-z|}{\sqrt{t}}\Big)^{\alpha\lambda}\Big(1+\frac{|z-y|}{\sqrt{s}}\Big)^{\alpha\lambda}\,dy\,dz\,ds\\
				&\lesssim M_\alpha^{\mathcal{L}}[f]\rho(x)^\alpha t^{-1} \int_0^{t/2} t^{-1} \Big(\frac{\sqrt{t}}{\rho(x)}\Big)^{2-\frac{n}{q}}\, ds \\ &\lesssim 
				M_\alpha^{\mathcal{L}}[f]\rho(x)^\alpha t^{-1}  \Big(\frac{\sqrt{t}}{\rho(x)}\Big)^\alpha,				\end{align*}
				where in the last inequality we use that $0 < \alpha \le 2-\frac{n}{q}$ and $t \le \rho(x)^2.$
				
				Similarly we proceed with for $B$. Again, by Lemmata   \ref{Cauchy},	\ref{Shenlemma} and 	\ref{estV}, we obtain		
				
				\begin{align*}
				|B |&\lesssim    \int_{t/2}^{t} \int_{\mathbb{R}^n}\int_{\mathbb{R}^n}  \omega_{t-s}(x-z)V(z)s^{-1} \omega_s(z-y) |f(y) | \rho(y)^{-\alpha}\rho(x)^\alpha\Big(1+\frac{|x-y|}{\rho(x)}\Big)^{\alpha\lambda}\,dy\,dz\,ds\\
				&\lesssim M_\alpha^{\mathcal{L}}[f]\rho(x)^\alpha t^{-1}\\
				&\quad \times 
				\int_{t/2}^t \int_{\mathbb{R}^n}\int_{\mathbb{R}^n} \omega_{t-s}(x-z)V(z) \omega_t(z-y) \Big(1+\frac{|x-z|}{\sqrt{t-s}}\Big)^{\alpha\lambda}\Big(1+\frac{|z-y|}{\sqrt{t}}\Big)^{\alpha\lambda}\,dy\,dz\,ds\\
				&\lesssim M_\alpha^{\mathcal{L}}[f]\rho(x)^\alpha t^{-1} \int_{t/2}^{t} (t-s)^{-1} \Big(\frac{\sqrt{t-s}}{\rho(x)}\Big)^\alpha\, ds \\ &\lesssim 
				M_\alpha^{\mathcal{L}}[f]\rho(x)^\alpha t^{-1+\alpha/2}.			\end{align*}

				Similarly, by the same arguments we have 
				\begin{align*}
				|E| &\lesssim \int_{\mathbb{R}^n} \int_{\mathbb{R}^n} \tilde{W}_{t/2}(x-z) V(z)W_{t/2}(z,y) |f(y)| \rho(y)^{-\alpha} \rho(x)^\alpha
				\Big(1+\frac{|x-y|}{\rho(x)}\Big)^{\lambda \alpha} \,dy\, dz \\
				&\lesssim M_\alpha^{\mathcal{L}}[f]\rho(x)^\alpha \int_{\mathbb{R}^n}\tilde{W}_{t/2} (x-z)V(z)\Big(1+\frac{|x-z|}{\sqrt{t}} \Big)^{\lambda \alpha} dz \\
				&\lesssim M_\alpha^{\mathcal{L}}[f]\rho(x)^\alpha t^{-1} \Big(\frac{\sqrt{t}}{\rho(x)} \Big)^\alpha\\
				&\lesssim M_\alpha^{\mathcal{L}}[f] t^{-1+\alpha/2}.
				\end{align*}
		 		
				If $t \ge \rho(x)^2$, then
				\begin{align*}
				|\partial_t& \tilde{W}_tf(x)-{\partial_t}{W}_tf(x)| \lesssim \int_{\mathbb{R}^n} t^{-1} \omega_t(x-y) |f(y)| \rho(y)^{-\alpha} \rho(x)^\alpha
				\Big(1+\frac{|x-y|}{\rho(x)}\Big)^{\lambda\alpha}dy \\& \lesssim   M_\alpha^{\LL}[f] t^{-1}\,\Big[\int_{|x-y|<\rho(x)} \omega_t(x-y) \rho(x)^\alpha \,dy + \int_{|x-y| >\rho(x)} \omega_t(x-y) \Big(\frac{|x-y|}{\sqrt{t}}\Big)^\alpha t^{\alpha/2} \, dy \Big] \\ & \lesssim  M_\alpha^{\LL}[f] t^{-1} (\rho(x)^\alpha+ t^{\alpha/2}) \lesssim M_\alpha^{\LL}[f] t^{-1+\alpha/2}.
				\end{align*}			\end{proof}

			\subsection{Proof of Theorem \ref{identities}}
				As a  consequence of the previous results we have the following theorem.
			\begin{theorem}\label{Holdercomparacion}
				For $0<\alpha \le  2-n/q$, a measurable function  $f \in \Lambda_{\alpha/2}^W$  if, and only if, $f \in \Lambda_{\alpha/2}^{\tilde{W}}$  and $\rho(\cdot)^{-\alpha} f(\cdot) \in L^\infty(\R^n).$ 
			\end{theorem}
			
			This result together with Theorem \ref{nuevoStein} is the last step of the proof of Theorem \ref{identities}.

	  \section{Applications. Proofs of Theorems \ref{Schau}, \ref{Holderestimates}, \ref{TRiesz}, and \ref{multiplicador}. }\label{pruebaaplicaciones}
	  
	\begin{lemma}\label{Potencial} Let  $\beta >0$ and  $\mathcal{T}_\beta$ be either the operator $(Id+\LL)^{-\beta/2}$ or the operator $ \LL^{-\beta/2}$. If $f$ is a function such that  $\rho(\cdot)^{-\alpha}f\in L^\infty(\R^n)$ for some $\alpha>0$, then  $\mathcal{T}_\beta f(x) $ is well-defined and satisfies $$M_{\alpha+\beta}^\LL[ \mathcal{T}_\beta f] \le C M_\alpha^{\LL}[f].$$ 

Moreover if $f \in L^\infty(\mathbb{R}^n)$ then $\mathcal{T}_\beta f(x) $ is well defined and 
$$M_{\beta}^{\LL}[ \mathcal{T}_\beta f] \le C \|f\|_\infty.$$ 
\end{lemma}

\begin{proof}
If $\rho(\cdot)^{-\alpha}f\in L^\infty(\R^n)$ for some $\alpha>0$, then by   Lemma \ref{sizeheat} we get 
\begin{align*}
	  |(Id+\LL)^{-\beta/2} f(x)|&=\Big| \frac{1}{\Gamma(\beta/2)}\int_0^\infty e^{-t}e^{-t\LL}f(x) t^{\beta/2}\frac{dt}{t}\Big|\\&\le C\, M_\alpha^{\LL}[f]\,   \int_0^{\rho(x)^2} \rho(x)^\alpha t^{\beta/2}\frac{dt}{t}+C\,M_\alpha^{\LL}[f]\,\int_{\rho(x)^2}^\infty \rho(x)^\alpha \left(\frac{\rho(x)^2}{t}\right)^{\beta/2+1}t^{\beta/2}\frac{dt}{t}\\ &= C\,M_\alpha^{\LL}[f]\,\rho(x)^{\alpha+\beta}, \: x\in\R^n.
	  \end{align*}
	  The same estimate works for  $ \LL^{-\beta/2}f.$ The proof in the second case runs parallel, since  Lemma \ref{sizeheat} has an obvious version for bounded functions.
	  
	  \end{proof}


	  {\it Proof of Theorem \ref{Schau}}. We prove only (i), estimate (ii) can be proved analogously. 
	  
   Let $f\in \Lambda^W_{\alpha/2}$. Lemma \ref{sizeheat} with $\ell=0$ together with Fubini's theorem allow us to get $\displaystyle W_y((Id+\LL)^{-\beta/2}f)(x)= 
	 \frac{1}{\Gamma(\beta/2)}\int_0^\infty e^{-t} W_y(W_tf)(x) t^{\beta/2}\frac{dt}{t}$.   Also observe that by the semigroup property and Lemma \ref{cambiaryrho} with $j=1$ and $m$ such that $ {\frac{m}{2}+1\ge[\alpha/2] +1  }  $,  we have 
	 \begin{align*}
 \int_0^\infty \Big|e^{-t} \partial_yW_y(W_tf)(x) \Big|t^{\beta/2}\frac{dt}{t}
	 &=  \int_0^\infty \Big| e^{-t} \partial_wW_w f(x)\big|_{w=y+t} \Big|t^{\beta/2}\frac{dt}{t}\\ & \le C\, S_\alpha[f]\, \int_0^\infty e^{-t} \rho(x)^m (y+t)^{-(m/2+1)+\alpha/2} t^{\beta/2}\frac{dt}{t}.
	 \end{align*}
The last integral can be bounded by a uniform (in a neighborhood of $y$) integrable function  (of $t$).  This means that  we can interchange the derivative with respect to $y$ and the integral with respect to $t$ in the above expression.
	
	  Let $\ell=[\alpha/2+\beta/2]+1$.  By iterating  the above arguments and using the hypothesis we have 
	  	\begin{align*}
	  	|\partial_y^\ell W_y ((Id+\LL)^{-\beta/2} f(x))|&=\left|\frac{1}{\Gamma(\beta/2)}\int_0^\infty e^{-t}\partial_y^\ell W_y(W_tf)(x) t^{\beta/2}\frac{dt}{t}\right|\\
	  	&\le {C\,S_\alpha[f]\,}\int_0^\infty e^{-t}(\partial_w^\ell W_w f(x)\Big|_{w=y+t})t^{\beta/2}\frac{dt}{t}
		\\&\le {C\,S_\alpha[f]\,}\int_0^\infty e^{-t}(y+t)^{-\ell+\alpha/2}t^{\beta/2}\frac{dt}{t}
	  	\\&\stackrel{\frac{t}{y}=u}{\le }C\,S_\alpha[f]\,{y^{\alpha/2+\beta/2-\ell}}\int_0^\infty \frac{u^{\beta/2}e^{-yu}}{(1+u)^{\ell-\alpha/2}}\frac{du}{u}
	  	\\&\le C\,S_\alpha[f]\, y^{\alpha/2+\beta/2-\ell}.
	  	\end{align*}
	  	
	  	 When $f\in L^\infty(\R^n)$ we apply Lemma \ref{Cauchy} and we get for $\ell= [\beta/2]+1$ that $|\partial_y^\ell W_y W_\nu f(x)| \le C \frac{\|f\|_{\infty}}{y^{\ell}}$. Then we can proceed as before.
		 
		 By using Lemma \ref{Potencial} we get the bound of $M_{\alpha+\beta}^{\LL}[f]$ and we end the proof of the  theorem.	  \edproof
	  
	  \begin{remark}
	  	 In the case $0<\alpha+\beta <1$, with $0<\alpha,\beta<1$, statement (i) of Theorem \ref{Schau} was obtained in  \cite{BonHar} and \cite{MSTZ}  for the spaces given by \eqref{polabruno}. Moreover,  (ii) is also proved for $0<\beta<1$ in \cite{BonHar}. 
	  	\end{remark}
\begin{lemma} \label{holderlema} Let $0< \beta < \alpha$ and $f$ be a function in the space $	\Lambda_{\alpha/2}^{W}$. Then 
$\mathcal{L}^{\beta/2} f$ is well defined and 
 $$M_{\alpha-\beta}^{\LL}[\mathcal{L}^{\beta/2} f] \le C_{\alpha,\beta} \| f\|_{\Lambda_{\alpha/2}^{W}}. $$
\end{lemma}

\begin{proof} We can write	\begin{align*}
	|\LL^{\beta/2}f(x)|= \left|\frac1{c_\beta} \left(\int_0^{\rho(x)^2} +\int_{\rho(x)^2}^\infty\right)
	(Id-e^{- t\LL})^{[\beta/2]+1}f(x)\,\frac{dt}{t^{1+\beta/2}}\right|=|I+II|.
		\end{align*}
		As $\rho(\cdot)^{-\alpha}f\in L^\infty(\R^n)$, by Lemma   \ref{sizeheat} we have 
	$$|II|\le C M_\alpha^\LL[f]\int_{\rho(x)^2}^\infty \rho(x)^\alpha\frac{dt}{t^{1+\beta/2}}=CM_\alpha^\LL[f]\rho(x)^{\alpha-\beta}.$$

	Now we shall estimate  $|I|$. Let $\ell=[\beta/2]+1$,
	by the semigroup property we have \newline$\displaystyle	{|}(Id-e^{- t\LL})^{[\beta/2]+1}f(x)	{|}={\Big|}\underbrace{\int_0^{t} \dots \int_0^t}_{\substack{\ell}} \partial_{y_1}\dots \partial_{y_\ell} W_{y_1+\dots + y_\ell} f(x) dy_1\dots dy_\ell{\Big|}$.

If $\beta/2<\alpha/2<\ell$, then  $k:=[\alpha/2]+1=\ell$ and
$$|	(Id-e^{- t\LL})^{\ell}f(x)|\le C\,S_\alpha^{W}[f]\underbrace{\int_0^{t} \dots \int_0^t}_{\substack{\ell}}\frac{dy_\ell\dots dy_1}{(y_1+\dots+y_{\ell})^{\ell-\alpha/2}}\le C\,S_\alpha^{W}[f]\,t^{\alpha/2}$$
so $\displaystyle|I|\le C \,S_\alpha^{W}[f]\,\int^{\rho(x)^2}_0 t^{\alpha/2}\frac{dt}{t^{1+\beta/2}}=C\,S_\alpha^{W}[f]\,\rho(x)^{\alpha-\beta}.$

If $\ell<\alpha/2$, then  $k>\ell$ and by Lemma \ref{cambiaryrho} we get, for $0<t\le \rho(x)^2$,
\begin{align*}|	(Id-e^{- t\LL})^{\ell}f(x)|&= \Big|\underbrace{\int_0^{t} \dots \int_0^t}_{\substack{\ell}}\Big(-\int_{y_1+\dots+y_\ell}^{\ell(\rho(x))^2}\frac{\partial_u^{\ell+1}W_uf(x)}{(\rho(x)^2)^{k-(\ell+1)}} du(\rho(x)^2)^{k-(\ell+1)}\\
&\quad\quad+\frac{\partial_\nu^\ell W_\nu f(x)\Big|_{\nu=\ell\rho(x)^2}}{(\rho(x)^2)^{k-\ell}}(\rho(x)^2)^{k-\ell}\Big)dy_1\dots dy_\ell\Big|\\
&\le  C\,S_\alpha^{W}[f](\rho(x)^2)^{k-(\ell+1)}\underbrace{\int_0^{t} \dots \int_0^t}_{\substack{\ell}}\int_{y_1+\dots+y_\ell}^{\ell(\rho(x))^2}u^{-k+\alpha/2}du {dy_1\dots dy_\ell}\\&
\quad+C\,S_\alpha^{W}[f]{(\ell\rho(x)^2)^{-k+\alpha/2}}(\rho(x)^2)^{k-\ell}\,t^\ell.
\end{align*}

Therefore, if $\alpha$ is not even we have, for $0<t\le \rho(x)^2$, 
\begin{align*}
|&	(Id-e^{- t\LL})^{\ell}f(x)|\\
&\le  C\,S_\alpha^{W}[f](\rho(x)^2)^{k-(\ell+1)}\underbrace{\int_0^{t} \dots \int_0^t}_{\substack{\ell}}((y_1+\dots+y_\ell)^{-k+\alpha/2+1}+(\ell\rho(x)^2)^{-k+\alpha/2+1}) {dy_1\dots dy_\ell}\\&\quad\quad
+C\,S_\alpha^{W}[f](\rho(x)^2)^{\alpha/2-\ell}t^\ell \\
&\le  C\,S_\alpha^{W}[f]((\rho(x)^2)^{k-(\ell+1)}t^{-k+\alpha/2+\ell+1}+(\rho(x)^2)^{-\ell+\alpha/2}t^\ell).
\end{align*}
Thus, in this case we get 
\begin{align*}\displaystyle|I|&\le C \,S_\alpha^{W}[f]\left((\rho(x)^{2})^{k-\ell-1}\int^{\rho(x)^2}_0t^{-k+\alpha/2+\ell-\beta/2}{dt}+(\rho(x)^2)^{-\ell+\alpha/2}\int^{\rho(x)^2}_0t^{\ell-\beta/2-1}{dt}\right)\\&=C\,S_\alpha^{W}[f]\rho(x)^{\alpha-\beta}.\end{align*}
If $\alpha$ is  even, then  $k=\alpha/2+1$  and,  for $0<t\le \rho(x)^2$, 
\begin{align*}
|({Id}-&e^{- t\LL})^{\ell}f(x)|\\
\le&
C\,S_\alpha^{W}[f](\rho(x)^2)^{\alpha/2-\ell}\underbrace{\int_0^{t} \dots \int_0^t}_{\substack{\ell}}(\log(\ell(\rho(x))^2)-\log(y_1+\dots+y_\ell)){dy_1\dots dy_\ell}\\
&\quad \quad+C\,S_\alpha^{W}[f](\rho(x)^2)^{\alpha/2-\ell}t^\ell.
\end{align*}	
In order to solve the last integral we can perform the change of variables  $\tilde{y}_1 =y_1 , \tilde{y}_2 =y_2 , \cdots , \tilde{y}_{\ell-1} =y_{\ell-1} ,\tilde{y} =y_1+\cdots+y_\ell $. Then we proceed as in the proof of Proposition \ref{tam}. 	
Putting together the above computations we get in this case
 $$
{\Big|} \int_0^{(\rho(x))^2} \frac{(Id-e^{- t\LL})^{\ell}f(x)}{t^{1+\beta/2}}dt{\Big|} \le C \,S_\alpha^{W}[f](\rho(x))^{\alpha-\beta}.
 $$
\end{proof}

 {\it Proof of Theorem \ref{Holderestimates}.} Let $\ell=[\beta/2]+1$ and $m= \left[\frac{\alpha-\beta}{2}\right]+1$. Then, $m+\ell=\left[\frac{\alpha-\beta}{2}\right]+1+[\beta/2]+1 >\alpha/2-\beta/2+\beta/2=\alpha/2$. As $m+\ell\in \mathbb{N}$ we get $m+\ell \ge[\alpha/2]+1.$
	  	
	  	By using the arguments in the proof of Lemma \ref{holderlema} we have
	  	\begin{align*}
	  	\Big|\partial_y^mW_y(\LL^{{\beta/2}} f)(x) \Big|&= \Big|{C_\beta} \int_0^\infty \partial_y^mW_y \Big(\underbrace{\int_0^{t} \dots \int_0^t}_{\substack{\ell}} \partial_\nu^\ell W_{\nu} |_{\nu= s_1+\dots+s_\ell} f(x)ds_1\dots ds_\ell \Big) \frac{dt}{t^{1+\beta/2}} \Big|\\ &=
	  	\Big|{C_\beta} \int_0^\infty \Big( \underbrace{\int_0^{t} \dots \int_0^t}_{\substack{\ell}}\partial_\nu^{m+\ell}W_{\nu} |_{\nu=y+ s_1+\dots+s_\ell}f(x) ds_1\dots ds_\ell \Big) \frac{dt}{t^{1+\beta/2}}\Big|
	  	\\ &\le
	  	C_\beta\,  S_\alpha^W[f] \int_0^\infty \Big(\underbrace{\int_0^{t} \dots \int_0^t}_{\substack{\ell}} (y+s_1+\dots s_\ell)^{-(m+\ell) +\alpha/2} ds_1\dots ds_\ell \Big) \frac{dt}{t^{1+\beta/2}} \\&=
	  	C_\beta\,  S_\alpha^W[f] \int_0^y ( \dots )  \frac{dt}{t^{1+\beta/2}}  + C_\beta\,    S_\alpha^W[f] \int_y^\infty  (\dots ) \frac{dt}{t^{1+\beta/2}} =C_\beta\,    S_\alpha^W[f]\,(A +B).
	  	\end{align*}
	   Now we shall estimate $A $ and $B$.
		  	\begin{align*} 
	  	A  &= 
	  	C_\beta y^{-m+\alpha/2} \int_0^y \underbrace{\int_0^{t/y} \dots \int_0^{t/y}}_{\substack{\ell}} (1+s_1+\dots s_\ell)^{-(m+\ell) +\alpha/2} ds_1\dots ds_\ell \frac{dt}{t^{1+\beta/2}}  \\&\le 
	  	C_\beta\, y^{-m+\alpha/2} \int_0^y   \Big(\frac{t}{y}\Big) ^\ell \frac{dt}{t^{1+\beta/2}} = C_\beta\, y^{-m+\alpha/2-\ell} \int_0^y  \frac{dt}{t^{1+\beta/2-\ell}}
	  	= C_\beta\,y^{-m+(\alpha-\beta)/2}.
	  	\end{align*}
	   On the other hand,
	  	\begin{align*}
	  	B  &\le  \int_y^\infty  \sum_{j=0}^\ell \frac{C_j}{(y+jt)^{m-\alpha/2}} \frac{dt}{t^{1+\beta/2}} =    \sum_{j=0}^\ell \int_y^\infty \frac{C_j}{(y+jt)^{m-\alpha/2}} \frac{dt}{t^{1+\beta/2}} \\& \le   \sum_{j=0}^\ell C_j y^{-m+(\alpha-\beta)/2}.
	  	\end{align*}
	  The last inequality is obtained by observing that $ y \le y+jt \le (1+\ell) t$ inside the integrals together with the  discussion  about the sign of $m-\alpha/2$.  	
	  	
	  	\edproof

	  \begin{remark}
	  	 The previous result was obtained  in \cite{MSTZ}  for the spaces given by \eqref{polabruno} when $0< \alpha < 1$.
	  	
	  	\end{remark}

	{\it Proof of Theorem \ref{TRiesz}.} Let $0<\alpha\le 1-n/q$ and $f\in 	\Lambda_{\alpha/2}^W$. By Theorem \ref{Schau} we have that $\LL^{-1/2}f \in 	\Lambda_{\frac{\alpha+1}{2}}^{W}$ and by Theorem \ref{Holdercomparacion} this means that  $\LL^{-1/2}f \in	\Lambda_{\frac{\alpha+1}{2}}^{\tilde{W}}$ and $\rho(\cdot)^{-(\alpha+1)}\LL^{-1/2}f\in L^\infty(\R^n)$. Therefore, by Proposition \ref{Steinderivadax}   we get that $\mathcal{R}_if=\partial_{x_i}(\LL^{-1/2}f)\in 	\Lambda_{\frac{\alpha}{2}}^{\tilde{W}}$ and $\rho(\cdot)^{-\alpha}\mathcal{R}_if\in L^\infty(\R^n)$.
	  	Thus, Theorem \ref{Holdercomparacion} gives the second statement of the theorem.

	  	Suppose now $1<\alpha\le2-n/q$ and $f\in 	\Lambda_{\alpha/2}^W$.  By Theorem \ref{Holdercomparacion} this means that  $f \in	\Lambda_{\frac{\alpha}{2}}^{\tilde{W}}$ and $\rho(\cdot)^{-\alpha}f\in L^\infty(\R^n)$.  Then, Proposition \ref{Steinderivadax}  gives that $\partial_{x_i}f\in 	\Lambda_{\frac{\alpha-1}{2}}^{\tilde{W}}$ and $\rho(\cdot)^{-(\alpha-1)}\partial_{x_i}f\in L^\infty(\R^n)$. Again,  by Theorem \ref{Holdercomparacion} this means that  $\partial_{x_i}f \in	\Lambda_{\frac{\alpha-1}{2}}^{{W}}$ and by Theorem \ref{Schau} we get that $R_if=\LL^{-1/2}(\partial_{x_i}f)\in \Lambda_{\frac{\alpha}{2}}^{{W}}$.

	  \edproof
	 
	 \begin{remark}
	 	Theorem \ref{TRiesz} was known in the case $0< \alpha <1$ for the spaces given by \eqref{polabruno}, see \cite{BonHar2}.
	 	
	 \end{remark}

	  {\it Proof Theorem \ref{multiplicador}.} Lemmas \ref{sizeheat} and \ref{cambiaryrho}   guaranty the integrability of $\partial_s(W_sf(x))$  as a function of $s.$ Then, we can write
	\begin{align*}  
	m(\LL) f(x) =  \int_0^\infty (-\partial_s(W_sf(x))a(s)ds  = \left(\int_{0}^{\rho(x)^2}+\int_{\rho(x)^2}^\infty \right){(-\partial_s(W_sf(x)))} a(s) ds=I+II.
	  \end{align*}
	   By using Lemma \ref{sizeheat}, we get  
	   \begin{align*}
	 \big|  II  \big|&\le C\|a\|_\infty\, M_\alpha^{\LL}[f]\rho(x)^\alpha\int_{\rho(x)^2}^\infty \frac{1}{s}\left(1+\frac{s}{\rho(x)^2}\right)^{-M} ds\\ &= C\|a\|_\infty \, M_\alpha^{\LL}[f] \,  \rho(x)^\alpha\int_1^\infty \frac{1}{u(1+u)^M}du\le  C \|a\|_\infty\, M_\alpha^{\LL}[f]\,  \rho(x)^\alpha.
	   \end{align*}
	  Now we estimate $I$. Let  $k=[\alpha/2]+1$. If $\alpha$ is not even, by Lemma \ref{cambiaryrho} we get
	  \begin{align*}
	  \big| I \big|&{\le}(\rho(x))^{2(k-1)}\int_{0}^{\rho(x)^2} \frac{|\partial_s W_sf(x)|}{(\rho(x))^{2(k-1)}}|a(s)|ds\le C\|a\|_\infty\, S_\alpha^{W}[f] \,  (\rho(x))^{2(k-1)}\int_{0}^{\rho(x)^2} s^{-k+\alpha/2}ds\\&=C\|a\|_\infty\, S_\alpha^{W}[f] \, \rho(x)^{\alpha}.
	  \end{align*}
	 If  $\alpha$ is even, by Lemma \ref{cambiaryrho} we have 
	 \begin{align*}
	 |I|&=\left|\int_{0}^{\rho(x)^2} \left(\int_{s}^{\rho(x)^2}\frac{\partial_u^2W_uf(x)}{(\rho(x))^{2(k-2)}}du (\rho(x))^{2(k-2)}- \frac{\partial_\nu W_\nu f(x)\Big|_{\nu=\rho(x)^2}}{(\rho(x))^{2(k-1)}} (\rho(x))^{2(k-1)}  \right)a(s)ds\right|\\
	 &\le C\|a\|_\infty  S_\alpha^{W}[f] \int_{0}^{\rho(x)^2} \left(\int_{s}^{\rho(x)^2}u^{-1}du (\rho(x))^{\alpha-2}+(\rho(x))^{\alpha-2}\right) ds\\
	 &= C\|a\|_\infty S_\alpha^{W}[f] \left(\int_{0}^{\rho(x)^2}(\log(\rho(x)^2)-\log s)(\rho(x))^{\alpha-2} ds+\rho(x)^{\alpha}\right)\\
	 &=C\|a\|_\infty  S_\alpha^{W}[f] \rho(x)^{\alpha}.
	 \end{align*}
Up to now, we have shown that $M_\alpha^{\LL}[m(\mathcal{L})f] \le C\|f\|_{\Lambda^{W}_{\alpha/2}}.$ 	 
	
	 Now we want to see that $\|\partial_u^kW_ym(\LL )f\|_\infty\le Cy^{-k+\alpha/2}$.   Fubini's Theorem together with Lemmas \ref{cambiaryrho} and \ref{sizeheat} allow us to interchange integral with derivatives and kernels. Then,	 \begin{align*}
	 |\partial_u^kW_ym(\LL )f(x)|&=\left|\int_0^\infty \partial_u^{k+1}W_uf(x)\big|_{u=y+s}a(s)ds\right|\le C\|a\|_\infty S_\alpha^{W}[f] \int_0^\infty \frac{ds}{(y+s)^{k+1-\alpha/2}}ds\\
	 &=C\|a\|_\infty S_\alpha^{W}[f]\, y^{-(k+1)+\alpha/2}\int_0^\infty\frac{ y}{(1+r)^{k+1-\alpha/2}}dr=C\|a\|_\infty S_\alpha^{W}[f] y^{-k+\alpha/2}.
	  \end{align*}
	 \edproof
		  
		  \begin{remark}
	In the case $0<\alpha <1$,  the previous result  was obtained in \cite{MSTZ} for the spaces given by \eqref{polabruno}.
		  	
		  \end{remark}
				  			
\section{Lipschitz spaces via the Poisson Semigroup} \label{secpoisson}

 
  The Poisson semigroup of the operator $\LL$ was  defined in 
 ( \ref{Poissonformula}).
		The following result was proved in \cite{MSTZ}. 	 
					 
					 \begin{lemma}\label{Poisson est}
Given $k \in \mathbb{N}$,  for any 
 $N>0$ there exists a constant $C=C_{N,k}$ such that
\begin{enumerate}[(a)]
    \item $\displaystyle  P_y(x,z)
            \leq C \frac{y}{(\abs{x-z}^2+y^2)^{\frac{n+1}{2}}}
             \left(1+\frac{(\abs{x-z}^2+y^2)^{1/2}}{\rho(x)}
             + \frac{(\abs{x-z}^2+y^2)^{1/2}}{\rho(z)}\right)^{-N}$;
    \item $\displaystyle| \partial_y^k P_y(x,z)|\leq C \frac{1}{(\abs{x-z}^2+y^2)^{\frac{n+k}{2}}}
          \left(1+\frac{(\abs{x-z}^2+y^2)^{1/2}}{\rho(x)}
          + \frac{(\abs{x-z}^2+y^2)^{1/2}}{\rho(z)}\right)^{-N}$.
\end{enumerate}
\end{lemma}
As a consequence, we have the following proposition.
\begin{proposition}\label{Poisoninfinito}
 Let $f$ be a function such that $M^P[f]=\int_{\R^n}\frac{|f(x)|}{(1+|x|)^{n+1}}dx<\infty.$ Then, $  \lim_{y\to \infty} \partial_y^\ell P_yf(x) = 0,$  \hbox{ for every } $\ell\in\N\cup\{0\}, \: x\in\R^n,$ and $\lim_{y\to 0}P_yf(x)=f(x),$ a.e. $x\in\R^n$.
\end{proposition}					 
\begin{proof}
The convergence to $0$ of the Poisson semigroup and its derivatives follows directly from the previous Lemma. It remains to prove that $\lim_{y\to 0}P_yf(x)=f(x),$ a.e. $x\in\R^n$.

By Lemma  \ref{Poisson est} we have that, for $y<1$,  
\begin{align}\label{Poissonlejos}  \int_{|x-z|> 2|x|}& |P_y(x,z)f(z)|dz\le \int_{|x-z|> 2|x|} \frac{y}{(|x-z|+y)^{n+1}} |f(z)| dz \\  &\le C\int_{|z|<1}  \frac{y}{(2|x|+y)^{n+1}} |f(z)| dz +C \int_{|z|>1} \frac{y}{(\frac{2}{3}|z|+y)^{n+1}} |f(z)| dz \nonumber\\ &\le
\frac{C y}{|x|^{n+1}}  \int_{|z|<1}  |f(z)| dz + Cy \int_{|z|>1} \frac{1}{(|z|+1)^{n+1}} |f(z)| dz\rightarrow 0, \hbox{ as } y\to 0. \nonumber
\end{align}
To manipulate the other integral, we proceed as in the proof of Lemma \ref{claim}. We compare the Poisson kernel with the kernel of the classical Poisson semigroup, $e^{-y\sqrt{-\Delta}}$, that we will denote by $\tilde{P}_y$.

By using \eqref{difcalor} we have that 
\begin{align*}
&\left|\int_{|x-z|<2|x|}(P_y(x,z)-\tilde{P}_y(x-z))f(z)dz \right|\\&\le C\int_0^\infty ye^{-\frac{y^2}{4\tau}}\int_{|x-z|<2|x|}|W_\tau(x,z)-\tilde{W}_\tau(x-z)||f(z)|dz \frac{d\tau}{\tau^{3/2}}\\&\le C\int_0^{\rho(x)^2} ye^{-\frac{y^2}{4\tau}}\int_{|x-z|<2|x|}\left(\frac{\sqrt{\tau}}{\rho(x)} \right)^{2-n/q}w_\tau(x-z)|f(z)|dz\frac{d\tau}{\tau^{3/2}}\\&\quad +C y\int_{\rho(x)^2}^\infty \frac{d\tau}{\tau^{3/2}}\int_{|x-z|<2|x|}|f(z)|dz\\&
\le  \frac{C(x)}{\rho(x)^{\epsilon}}\int_0^{\rho(x)^2} \frac{y}{\tau^{1/2}}e^{-\frac{y^2}{4\tau}}({\sqrt{\tau}})^{\epsilon}\frac{d\tau}{\tau}+ \frac{C(x) y}{\rho(x)}\\
&\le\frac{C(x)y^\epsilon}{\rho(x)^{\epsilon}}+\frac{C(x) y}{\rho(x)}\rightarrow 0, \text{ as } y\to 0,
\end{align*}
where $0<\epsilon<1$.

Finally, by the point-wise convergence of the classical Poisson semigroup to $L^1$ functions, we deduce the result.
\end{proof}
 	
Parallel to the heat semigroup case,  in order to prove Theorem \ref{tam3}, we shall need this lemma.

\begin{lemma}\label{cambiaryrhoP}
					Let $\alpha>0$ and $k=[\alpha]+1$. Assume that $f \in \Lambda_{\alpha}^{{P}}$, then for every $j,m\in\N\cup \{0\}$ such that $m+j\ge k$,  there exists a $C_{m,j}>0$ such that
					$$
					\left\| \frac{\partial_y^j P_yf}{\rho(\cdot)^m}\right\|_{\infty}\le C_m	S^P_\alpha[f] y^{-(m+j)+\alpha}.
					$$
				\end{lemma}
				\begin{proof}  For  $\ell > k$, by the semigroup property and Lemma \ref{Poisson est} we get that
					\begin{align*}
					\left|\frac{\partial_y^\ell P_yf(x)}{\rho(x)^m}\right|&=\left|\frac{C_\ell }{\rho(x)^m}\int_{\R^n} \partial_v^{\ell-k}P_{v}(x,z)|_{v=y/2}\partial_u^kP_uf(z)|_{u=y/2}dz\right|\\
					&\le \frac{C_\ell \|\partial_u^k P_uf|_{u=y/2}\|_\infty}{\rho(x)^m}\int_{\R^n} \frac{1}{(|x-z|^2+y^2)^{\frac{n+\ell-k}{2}}}\left(\frac{\rho(x)}{y} \right)^{m}dz\\
					&\le  C_\ell	S^P_\alpha[f]y^{-(m+\ell)+\alpha}, \:\; x\in\R^n.
					\end{align*}

						If  $j\le k$, since the $y-$derivatives of $P_yf(x)$ tend to zero as $y\to \infty$, we  integrate $\ell-j$ times the previous estimate  and we get the result.

				\end{proof}

				{\it Proof  of Theorem  \ref{tam3}.}
					By using Proposition \ref{Poisoninfinito} we have 
					\begin{align*}
					|f(x)| &\le  \sup_{0< y < \rho(x)} |P_yf(x) | \\ &\le 
					\sup_{0< y < \rho(x)} |P_yf(x)-P_{\rho(x)}f(x) | + 
					|P_{\rho(x)}f(x)| \\ &= I+II.
					\end{align*}	
					Let $k = [\alpha]+1$. By using Lemma \ref{cambiaryrhoP} with $j=0$ and $m = k$ we have 
					\begin{align*}
					II=|P_{\rho(x)}f(x)|= \left|\frac{P_{\rho(x)}f(x)}{\rho(x)^{k}} \right|\rho(x)^{k}\le C 	S^P_\alpha[f] (\rho(x))^{-k+\alpha}\rho(x)^{k}=C\,	S^P_\alpha[f]\rho(x)^\alpha.
					\end{align*}
					
					Now we shall estimate  $I$. If $\alpha$ is not integer, by Lemma \ref{cambiaryrhoP} with $j=1$ and $m =k-1$  we have that
					\begin{align*} 
					I &\le   \rho(x)^{k-1} \sup_{0< y < \rho(x)} \int_y^{\rho(x)} \left|\frac{\partial_z P_ z f(x)}{\rho(x)^{k-1}}\right| dz   \le C\,
						S^P_\alpha[f] \rho(x)^{k-1} \sup_{0< y < \rho(x)} \int_y^{\rho(x)}  z^{-k+\alpha} dz  \\ &\le 
					C\,	S^P_\alpha[f]\rho(x)^{k-1}   \sup_{0< y < \rho(x)} ((\rho(x))^{-(k-1)+\alpha}- y^{-(k-1)+\alpha}) \le C\,	S^P_\alpha[f] \rho(x)^{\alpha}.
					\end{align*}
					When $\alpha$ is an integer, we  write 
					\begin{align*}
					I&= \sup_{0< y < \rho(x)}\left| \int_y^{\rho(x)}{\partial_z P_ z f(x)}dz\right|\\&=\sup_{0< y < \rho(x)}\left| \int_y^{\rho(x)}\left(-\int_{z}^{\rho(x)}\partial_u^2 P_ u f(x)du+\partial_v P_v f(x)|_{v=\rho(x)}\right)dz\right|.
					\end{align*}		
					By Lemma \ref{cambiaryrhoP} with $j=2$ and $m= k-2$, since $k= \alpha+1$,  we get
					\begin{align*}
					\Big| &\int_y^{\rho(x)}\int_{z}^{\rho(x)}\partial_u^2 P_ u f(x)dudz\Big|=\rho(x)^{k-2}\left| \int_y^{\rho(x)}\int_{z}^{\rho(x)}\frac{\partial_u^2 P_ u f(x)}{\rho(x)^{k-2}}dudz\right|\\
					&\le C\,	S^P_\alpha[f]\rho(x)^{\alpha-1} \int_y^{\rho(x)}\int_{z}^{\rho(x)}u^{-1}dudz= C\,	S^P_\alpha[f] \rho(x)^{\alpha-1} \int_y^{\rho(x)}(\log(\rho(x))-\log z)dz\\
					&= C\,	S^P_\alpha[f]\rho(x)^{\alpha-1}\big[\log(\rho(x))(\rho(x)-y)-(\rho(x)\log(\rho(x))-\rho(x)-y\log y+y)\big]\\
					&=C\,	S^P_\alpha[f] \rho(x)^{\alpha-1}\big[y\log \big(\frac{y}{\rho(x)}\big)+\rho(x)-y\big]\le C\,	S^P_\alpha[f] \rho(x)^{\alpha}.
					\end{align*}	
					For the second summand of $I$, Lemma \ref{cambiaryrhoP}, with $j=1$ and $m= k-1$ applies, 	 so 
					\begin{multline*} 
					\sup_{0< y < \rho(x)}	(\rho(x)-y)|\partial_v P_v f(x)|_{v=\rho(x)}|=	\sup_{0< y < \rho(x)}	(\rho(x)-y)\rho(x)^{k-1}\frac{|\partial_v P_v f(x)|_{v=\rho(x)}|}{\rho(x)^{k-1}}\\ \le 	C\,	S^P_\alpha[f]\sup_{0< y < \rho(x)}	(\rho(x)-y)\rho(x)^{\alpha} (\rho(x))^{-1} \le C\,	S^P_\alpha[f]\rho(x)^{\alpha}.	
					\end{multline*}

		\edproof

			To prove Theorem \ref{identities4}, we need to define an auxiliary class of Lipschitz functions by means of the classical  Poisson semigroup,  $\tilde{P}_y=e^{-y\sqrt{-\Delta}}$. Again, the crucial difference between this class and the one defined by Stein in \cite{Stein} is that the functions don't need to be bounded.

		We define  $\Lambda_{\alpha}^{\tilde{P}}$  as the collection of functions satisfying   $M^P[f] < \infty$ and
			\begin{equation*} \Big\|\partial_y^k{\tilde{P}}_y f \Big\|_{L^\infty(\mathbb{R}^{n})}\leq C_\alpha y^{-k+\alpha},\;\: \, {\rm with }\, k=[\alpha]+1, y>0. 
			\end{equation*} 
			We denote by  $S^{\tilde{P}}_\alpha[f]$ as the infimum of the constants $C_\alpha$ above.
			
	
	\begin{remark}\label{goingtoinfinityP} Observe that the space $\Lambda_{\alpha}^{\tilde{P}}$ is well defined, because if $f$ is a function such that $M^P[f] < \infty$, then 
		\begin{itemize}
			\item[(i)] $|\partial_{x_i}^m\partial_y^\ell\tilde{P}_yf(x)|\to 0$ as $y\to \infty$ as far as $m+\ell \ge k\ge 1$.   
			Indeed, 
			\begin{align*}
			|\partial_{x_i}^m\partial_y^\ell\tilde{P}_yf(x)| &\le C\int_{|x-z|<|x|} \frac{ |f(z)| }{(|x-z|+y)^{n+k}}dz +C\int_{|x-z|>|x|}  \frac{ |f(z)| }{(|z|+y)^{n+k}}  dz\\ &\le  C\frac1{y^{n+k}} \int_{|x-z|<|x|}  |f(z)| dz +C\int_{|x-z|>|x|}  \frac{ |f(z)| }{(|z|+y)^{n+k}}  dz. 
			\end{align*}
			\noindent Both summands tend to zero, the second one by dominated convergence.
			\item[(ii)]  $ \lim_{y\to 0} \tilde{P}_yf(x) = f(x)$   a.e. $x\in\R^n$. This can be proved as we did in \eqref{Poissonlejos} and by using the $a.e.$ convergence of the classical Poisson semigroup for $L^1$ functions. 
		\end{itemize}		
	\end{remark}

		Moreover, we can prove the following results  analogously as we did for the heat semigroup.
	
		\begin{proposition}\label{subirelkP}
			Let {$\alpha>0$}, $k=[\alpha]+1$ and $f$ be a function satisfying  $M^P[f] < \infty$. Then, $\| \partial_ y^{k} \tilde{P}_y f\|_ {L^\infty(\mathbb{R}^n)} \le C_k y^{-k+ \alpha}$ if, and only if,   for $m \ge k$, $\| \partial_ y^{m} \tilde{P}_y f\|_ {L^\infty(\mathbb{R}^n)} \le C_m y^{-m+ \alpha}$.
		\end{proposition}

		\begin{theorem}\label{calor-poisson}
				Let $\alpha>0$ and $f$ a function such that $M^P[f] < \infty$.  If $f\in \Lambda_{\alpha/2}^W$, then $f\in\Lambda_{\alpha}^P$. Moreover, $\displaystyle S_\alpha^P[f]  \le C S_\alpha^W[f] .$
			\end{theorem}
\begin{proof}
Let $k=[\alpha/2]+1$ and $f\in \Lambda_{\alpha/2}^W$,  then  $[\alpha]+1=[\alpha/2+\alpha/2]+1\le[\alpha/2]+[\alpha/2]+2=2k$. By Proposition \ref{subirelkP}  it is enough to prove that $\|\partial_y^{2k}P_yf\|_\infty\le Cy^{-(2k)+\alpha}$. 

Since $\partial_y^2\Big( \frac{ye^{-\frac{y^2}{4\tau}}}{\tau^{3/2}}\Big)=\partial_\tau \Big(\frac{ye^{-\frac{y^2}{4\tau}}}{\tau^{3/2}}\Big)$, $k$-times integration by parts give
			\begin{align*}
			|\partial_y^{2k}P_yf(x)|&=\left| \frac{1}{2\sqrt{\pi}}\int_0^\infty\partial_y^{2k}\left(\frac{ye^{-\frac{y^2}{4\tau}}}{\tau^{3/2}}\right)e^{-\tau \LL}f(x)d\tau\right|=	\left| \frac{1}{2\sqrt{\pi}}\int_0^\infty\partial_\tau^{k}\left(\frac{ye^{-\frac{y^2}{4\tau}}}{\tau^{3/2}}\right)e^{-\tau \LL}f(x)d\tau\right|\\
			&=\frac{1}{2\sqrt{\pi}} \left|\int_0^\infty(-1)^k\left(\frac{ye^{-\frac{y^2}{4\tau}}}{\tau^{3/2}}\right)\partial_\tau^{k}e^{-\tau \LL}f(x)d\tau\right| \le C \,S_\alpha^W[f] \int_0^\infty \frac{ye^{-\frac{y^2}{4\tau}}}{\tau^{3/2}}\tau^{-k+\alpha/2}d\tau\\
			&\le C\,S_\alpha^W[f] \left(\frac{1}{y^2}\int_0^{y^2} \frac{y^3}{\tau^{3/2}}e^{-\frac{y^2}{4\tau}} \tau^{-k+\alpha/2}d\tau+\int_{y^2}^\infty \tau^{-k+\alpha/2}\frac{d\tau}{\tau}\right)\\
			&\le C\,S_\alpha^W[f] y^{-2k+\alpha}.
			\end{align*}
			\end{proof}
			
			The following Lemma is parallel to Lemma \ref{derivX}. We leave the details of the proof to the interested reader.
		\begin{lemma}\label{cambiaryderx}
			Let $\alpha>0$ and $k=[\alpha]+1$. If $f\in \Lambda_{\alpha}^{\tilde{P}}$,  then for every $j,m\in\N\cup \{0\}$ such that ${m}+j\ge k$,  there exists a $C_{m,j}>0$ such that
			$$
			\left\| {\partial_{x_i}^m\partial_y^j \tilde{P}_yf}\right\|_{\infty}\le C\,S^{\tilde{P}}_\alpha[f]y^{-(m+j)+\alpha}, \text{ for every } i=1\dots,n.
			$$
		\end{lemma}

			\begin{theorem}\label{nuevoStein2}
				Let $0<\alpha < 2$. Then $f\in \Lambda_{\alpha}^{\tilde{P}}$, if and only $M^P[f]<\infty$ and 
				$$N_\alpha[f]=\sup_{|z|>0}\frac{\|f(\cdot+z)+f(\cdot-z)-2f(\cdot)\|_\infty}{|z|^\alpha} <\infty.$$
			\end{theorem}
			\begin{proof}
				Let  $x\in\R^n$. We can write, for every $y>0$, $z\in\R^n$,
					\begin{align*}
					|f(x+z)&+f(x-z)-2f(x)|\le |\tilde{P}_yf(x+z)-f(x+z)+\tilde{P}_yf(x-z)-f(x-z)\\&-2(\tilde{P}_yf(x)-f(x))|+ |\tilde{P}_yf(x+z)-\tilde{P}_yf(x) +\tilde{P}_yf(x-z)-\tilde{P}_yf(x)|\\&=A+B.
					\end{align*}
				By using Lemma \ref{cambiaryderx} we can proceed  as in the proof of Theorem \ref{nuevoStein}. We have  \begin{align*}B=|\tilde{P}_yf(x+z)-\tilde{P}_yf(x) +\tilde{P}_yf(x-z)-\tilde{P}_yf(x)|&\le C\,S^{\tilde{P}}_\alpha[f]y^{-2+\alpha}|z|^2, \:\: \end{align*}
		If $0<\alpha<1$, by using Remark \ref{goingtoinfinityP}  we have that
					$$
					|\tilde{P}_yf(x)-f(x)|=\left|\int_0^y \partial_u \tilde{P}_uf(x)du \right|\le C\,S^{\tilde{P}}_\alpha[f]\int_0^y u^{-1+\alpha}du =C\,S^{\tilde{P}}_\alpha[f]y^{\alpha},
					$$
					and the same for the other two summands of $A$. 
				
				If $1< \alpha<2$, by proceeding as in the proof of Theorem \ref{nuevoStein}, by Lemma \ref{cambiaryderx} we have that 
					\begin{align*}A&=\left|\int_0^y (\partial_u \tilde{P}_uf(x+z)+\partial_u \tilde{P}_uf(x-z)-2\partial_u \tilde{P}_uf(x)) du\right|\\ 	 	
					&=\left| \int_0^y\int_0^1(\nabla_w \partial_u\tilde{P}_uf({w})_{|_{w=x+\theta z}}\cdot z-\nabla_v \partial_u\tilde{P}_uf({v})_{|_{v=x-\theta z}}\cdot z) d\theta du\right|\\
					&\le  C\,S^{\tilde{P}}_\alpha[f]\int_0^yu^{-2+\alpha}du|z|\le C \,S^{\tilde{P}}_\alpha[f]y^{-1+\alpha}|z|. \end{align*}
Thus, by choosing $y=|z|$ in each case we get what we wanted. 

For $\alpha=1$, by using that
	$\partial_{u} \tilde{P}_{u}f(x) = -\int_{u}^y \partial^2_{w} \tilde{P}_{w}f(x)dw + \partial_{y} \tilde{P}_{y}f(x)$, we have
					\begin{align*}|A|&\le C  S^{\tilde{P}}_1[f]\int_0^y \int_u^y  w^{-1} dw du + \Big|\int_0^y  \Big((\partial_y \tilde{P}_yf(x+z)+\partial_y \tilde{P}_yf(x-z)-2\partial_y \tilde{P}_yf(x)) \Big) du  \Big|\\ 	 	&= A_1+A_2.
					\end{align*}
Observe that $A_1 \le C S^{\tilde{P}}_1[f]y$. Regarding $A_2$, we proceed as in the case $1<\alpha <2$ and we have 
					\begin{align*}
		A_2 \le  		&\left| y\int_0^1(\nabla_{\tilde{x}} \partial_y\tilde{P}_yf({\tilde{x}})_{|_{\tilde{x}=x+\theta z}}\cdot z-\nabla_{{\tilde{z}}} \partial_y\tilde{P}_yf({\tilde{z}})_{|_{\tilde{z}=x-\theta z}}\cdot z)d\theta \right| \le C\,S^{\tilde{P}}_1[f]\,|z|.	 	 \end{align*}
When  $y=|z|$  we get what we wanted. 

For the converse we proceed as in Theorem \ref{nuevoStein}.				
			
			\end{proof}
			
				\begin{theorem} Let  $0<\alpha\le {2-n/q}$ and $f$ be a function such that $M^P[f] < \infty.$
			If $\rho(\cdot)^{-\alpha} f(\cdot) \in L^\infty(\mathbb{R}^n)$,  then 
					$$\|\partial_y^2P_yf-\partial_y^2\tilde{P}_yf \|_\infty \le C  M_\alpha^{\LL}[f] y^{-2+\alpha}.$$
			\end{theorem}
			
			\begin{proof}
				By subordination formula, integration by parts and and Theorem \ref{comparacion} we have that
				\begin{align*}
				|\partial_y^2P_yf(x)-\partial_y^2\tilde{P}_yf(x)|&=\left|\frac{1}{2\sqrt{\pi}}\int_0^\infty \partial_y^2\left(\frac{ye^{-{\frac{y^2}{4 \tau}}}}{\tau^{3/2}}\right)(W_\tau f-\tilde{W}_{\tau}f){d\tau}\right|\\
				&=\left|\frac{1}{2\sqrt{\pi}}\int_0^\infty \partial_\tau\left(\frac{ye^{-{\frac{y^2}{4 \tau}}}}{\tau^{3/2}}\right)(W_\tau f-\tilde{W}_{\tau}f){d\tau}\right|\\
				&\le \frac{1}{2\sqrt{\pi}}\int_0^\infty \frac{ye^{-{\frac{y^2}{4 \tau}}}}{\tau^{3/2}}|\partial_\tau(W_\tau f-\tilde{W}_{\tau}f)|{d\tau}\\
				&\le C\,M_\alpha^{\LL}[f]\,\int_0^\infty \frac{ye^{-{\frac{y^2}{4 \tau}}}}{\tau^{3/2}}\tau^{-1+\alpha/2}{d\tau}\\
				&\le C\,M_\alpha^{\LL}[f]\,\left(\frac{1}{y^2}\int_0^{y^2} \frac{y^3}{\tau^{3/2}}e^{-\frac{y^2}{4\tau}} \tau^{-1+\alpha/2}d\tau+\int_{y^2}^\infty \tau^{-1+\alpha/2}\frac{d\tau}{\tau}\right)\\
				&				\le C M_\alpha^{\LL}[f] y^{-2+\alpha}.\\
				\end{align*}
				
				\end{proof}
				
	A consequence of the previous theorem is the following. 
	\begin{theorem} \label{identities3}	Let  $0<\alpha\le  2-n/q$  and $f$ be a function such that $M^P[f] < \infty$ and $\rho(\cdot)^{-\alpha} f(\cdot) \in L^\infty(\mathbb{R}^n).$ Then, $f\in \Lambda_{\alpha}^P$ 	if and only if  $f\in \Lambda_{\alpha}^{\tilde{P}}$.   
			
\end{theorem}

Finally it is easy to see that  Theorems  \ref{nuevoStein},  \ref{calor-poisson}, \ref{identities3} and \ref{nuevoStein2} have as a consequence that Theorem  \ref{identities4} is true.

	\end{document}